\documentclass[12pt,a4paper]{amsart}
\usepackage{amsmath}
\usepackage{amssymb}
\usepackage{amsthm}
\usepackage{latexsym}
\usepackage{pstricks}
\usepackage{graphicx}
\usepackage{wrapfig}
\usepackage{lscape}
\usepackage{rotating}
\usepackage{amsbsy}
\usepackage[all]{xypic}
\usepackage{amscd}
\usepackage{mathrsfs}
\usepackage{txfonts}
\usepackage{amsfonts}
\usepackage{graphicx}
\usepackage{listings}
\usepackage{url}
\usepackage{bookmark}
\usepackage{tikz}
\allowdisplaybreaks
\newtheorem{theorem}{Theorem}[section]
\newtheorem{lemma}[theorem]{Lemma}
\newtheorem{corollary}[theorem]{Corollary}
\newtheorem{proposition}[theorem]{Proposition}

\newtheorem{claim}[theorem]{Claim}
\theoremstyle{definition}
\newtheorem{definition}[theorem]{Definition}

\newcommand{\circlesign}[1]{ 
    \mathbin{
        \mathchoice
        {\buildcirclesign{\displaystyle}{#1}}
        {\buildcirclesign{\textstyle}{#1}}
        {\buildcirclesign{\scriptstyle}{#1}}
        {\buildcirclesign{\scriptscriptstyle}{#1}}
    } 
}

\newcommand\buildcirclesign[2]{%
    \begin{tikzpicture}[baseline=(X.base), inner sep=0, outer sep=0]
    \node[draw,circle] (X)  {\ensuremath{#1 #2}};
    \end{tikzpicture}%
}

\begin{document}

\providecommand{\Gen}{\mathop{\rm Gen}\nolimits}%
\providecommand{\cone}{\mathop{\rm cone}\nolimits}%
\providecommand{\Sub}{\mathop{\rm Sub}\nolimits}%
\providecommand{\coker}{\mathop{\rm coker}\nolimits}%

\def\A{\mathsf{A}}
\def\B{\mathsf{B}}
\def\C{\mathcal{C}}
\def\D{\mathcal{D}}
\def\E{\mathcal{E}}
\def\V{\mathcal{V}}
\def\W{\mathsf{W}}
\def\T{\mathcal{T}}
\def\P{\mathcal{P}}
\def\SS{\mathcal{S}}
\def\X{\mathcal{X}}
\def\Y{\mathcal{Y}}
\def\Z{\mathcal{Z}}
\def\K{\mathsf{K}}
\def\F{\mathcal{F}}

\def\U{\mathcal{U}}
\providecommand{\add}{\mathop{\rm add}\nolimits}%
\providecommand{\End}{\mathop{\rm End}\nolimits}%
\providecommand{\Ext}{\mathop{\rm Ext}\nolimits}%
\providecommand{\Hom}{\mathop{\rm Hom}\nolimits}%
\providecommand{\ind}{\mathop{\rm ind}\nolimits}%
\providecommand{\id}{\mathop{\rm id}\nolimits}%
\newcommand{\module}{\mathop{\rm mod}\nolimits}%

\newcommand{\RR}{\mathbb{R}}
\newcommand{\PP}{\mathbb{P}}

\newcommand{\YY}{B}
\newcommand{\Yi}{B_{X_i}}
\newcommand{\Ui}{U_{X_i}}
\newcommand{\UU}{U_{X_i}}

\title{A category of wide subcategories}

\author[Buan]{Aslak Bakke Buan}
\address{
Department of Mathematical Sciences, 
Norwegian University of Science and Technology,
7491 Trondheim,
NORWAY
}
\email{aslak.buan@ntnu.no}

\author[Marsh]{Robert J. Marsh}
\address{
School of Mathematics,
University of Leeds,
Leeds LS2 9JT UK
}
\email{marsh@maths.leeds.ac.uk}

\begin{abstract}
An algebra is said to be \emph{$\tau$-tilting finite} provided
it has only a finite number of $\tau$-rigid objects up to isomorphism.
We associate a category to each such algebra. The
objects are the wide subcategories of its category of finite dimensional
modules, and the morphisms are indexed by support $\tau$-tilting pairs.
\end{abstract}

\thanks{
This work was supported by FRINAT grant number 231000, from the
Norwegian Research Council. The work 
for this paper was done during several visits of A. B. Buan to Leeds in 2017-2018 and
he would like to thank R. J. Marsh and the School of Mathematics at the University of Leeds for their warm hospitality.}

\maketitle

\section*{Introduction and main result}

A full subcategory $\B$ of an abelian category $\A$ is called 
{\em wide} if it is an exact abelian subcategory, or equivalently it is closed under
kernels, cokernels and extensions.

Let $\Lambda$ be a finite dimensional algebra over a field $k$, and $\module \Lambda$ the category 
of finitely generated left $\Lambda$-modules. 
Let $\tau$ denote the Auslander-Reiten translate in $\module \Lambda$.
Following \cite{air}, we call a $\Lambda$-module $M$ with $\Hom(M, \tau M) = 0$ a \emph{$\tau$-rigid} module. The algebra $\Lambda$
is called \emph{$\tau$-tilting finite} \cite{dij} 
if there are only a finite number of isomorphism classes
of indecomposable $\tau$-rigid $\Lambda$-modules.
By \cite{air} this is equivalent to $\Lambda$ having finitely many isomorphism classes of basic $\tau$-tilting modules, as defined in \cite{air}. In particular, all 
algebras of finite representation type, as well as all preprojective algebras 
of Dynkin type are $\tau$-tilting finite~ \cite{miz}; see \cite{dij} for further examples.

For a module $U$, let $U^{\perp} = \{X \in \module \Lambda \mid  \Hom(U,X)= 0\}$, and define ${}^{\perp}U$ similarly.
Jasso~\cite{jasso} proved that, if $U$ is $\tau$-rigid, then
the subcategory $J(U) = U^{\perp} \cap {^{\perp}(\tau U)}$ is equivalent to
a module category, and by~\cite{dirrt} we have that $J(U)$ is a
wide subcategory of $\module \Lambda$.
For a wide subcategory $\W$ which is equivalent to a module category, and a
module $V$ which is $\tau$-rigid in $\W$, 
we let $J_{\W}(V) = V^{\perp} \cap {^{\perp}(\tau_{\W} V)} \cap \W$.
Note that the AR-translations $\tau$ in $\module \Lambda$ and $\tau_{\W}$
in $\W$ will usually be different.

Let $\C(\Lambda)=\C(\module \Lambda)$ be the full subcategory of 
the \sloppy  bounded derived category $D^b(\module \Lambda)$ with objects corresponding to $\module \Lambda \amalg (\module \Lambda)[1]$.
For a full subcategory $\Y$ of $\module\Lambda$, we shall denote by
$C(\Y)$ the full subcategory $\Y\amalg \Y[1]$ of $\C(\Lambda)$.
As in~\cite{bm}, we say $\U = U \amalg P[1]$ is \emph{support $\tau$-rigid} in $\C(\module\Lambda)$ if $U,P$ are modules, $P$ is projective, $U$ is $\tau$-rigid and $\Hom(P,U)= 0$.
Analogously, if $\W$ is a wide subcategory of $\module\Lambda$ equivalent to a module category, we will say that an object $\U = U \amalg P[1]$ in
$\C(\W)$, where $U,P\in \W$, the object $P$ is projective in $\W$, the object $U$ is $\tau$-rigid
in $\W$ and $\Hom(P,U)=0$, is support $\tau$-rigid in $\C(\W)$. We let $J(\U) = J(U) \cap P^{\perp}$. We then have the following.

\begin{theorem}\label{rig-fin}
Let $\Lambda$ be a finite dimensional algebra, then the following hold.
\begin{itemize}
\item[(a)] \cite[Thm. 3.28]{dirrt},~\cite[Thm. 3.8]{jasso} If $\U$ is support $\tau$-rigid in $\C(\module \Lambda)$, then the subcategory $J(\U)$ is wide, and it is equivalent to a module category of a finite dimensional algebra.
\item[(b)] \cite[Thm. 3.34]{dirrt} If $\Lambda$ is $\tau$-tilting finite, then any wide subcategory of $\module
 \Lambda$ is of the form $J(\U)$ for some support $\tau$-rigid object $\U$ in $\C(\Lambda)$.
\end{itemize}
\end{theorem}

%When $\Lambda$ is $\tau$-rigid finite, then by \cite{dirrt} all wide 
%subcategories are of the form $J(\U)$ for some $\tau$-rigid object $\U$ in $\C(\Lambda)$.
The aim of the paper is to prove the following result.

\begin{theorem}\label{main}
Assume $\Lambda$ is $\tau$-tilting finite.
Then there is a category $\mathfrak{W}_\Lambda$ whose objects are all wide subcategories of $\module \Lambda$ and such that the maps from $\W_1$ to $\W_2$ are indexed by
all basic $\tau$-rigid objects $T$ in $\C(\W_1)$ such that $\W_2 = J_{\W_1}(T)$. 
\end{theorem}

Our results are inspired by a recent paper of 
Igusa and Todorov \cite{it}, where they defined a similar category in the setting of hereditary finite dimensional algebras. 

In Section \ref{s:mainresult} we state the main results of the paper and explain
how they are used to prove Theorem \ref{main}.

%In the next section (Section~\ref{s:mainresult}) we give the proof this theorem assuming %some key properties of wide subcategories, $\tau$-rigid objects and their
%relationship to the map $U\mapsto J(U)$ which will be proved in the later
%sections.

\section{Key steps for the proof of the main result} \label{s:mainresult}
For a (skeletally small) Krull-Schmidt category $\mathsf{X}$, let $\ind \mathsf{X}$ denote the set of isomorphism classes of indecomposable objects
in $\mathsf{X}$ and for any basic object $X$ in $\mathsf{X}$ let $\delta(X)$ denote the number of indecomposable direct summands of $X$. We generally assume all objects are basic and
we always assume subcategories are full and closed under isomorphism.

Firstly, we need the following, which is a generalization of \cite[Propositions 5.6 and 5.10]{bm}, and can be seen as a refinement of \cite[Theorem 3.15]{jasso}.
This is crucial. 

\begin{theorem}[Theorem \ref{thm-bi}]\label{main-bi}
Let $\U$ be a support $\tau$-rigid object in $\C(\Lambda)$.
Then there are bijections 
$$\{\X \in \ind(\C(\Lambda)) \mid \X \amalg \U \text{ } \tau\text{-rigid}\}
\setminus \ind \U $$ 

$$ \E_{\U} \downarrow \text{  } \uparrow \E_{\U} $$

$$\{\X \in  \ind(\C(J(\U)) \mid \text{$\X$ is support $\tau$-rigid in $\C(J(\U))$} \}.$$
\end{theorem}

The map $\E_{\U}$ can be extended additively, giving the following:

\begin{theorem}[Theorem \ref{rigid-sums}]\label{main-biadd}
Let $\U$ be a support $\tau$-rigid object in $\C(\Lambda)$ with $\delta(\U) = t'$.
For any positive integer $t \leq n- t'$, 
the map $\E_{\U}$ induces a bijection between:
\begin{itemize}
\item[(a)] The set of support $\tau$-rigid objects $\X$ in $\C(\Lambda)$ such that
$\delta(\X) = t$, the object $\X \amalg \U$ is support $\tau$-rigid and $\add \X \cap \add \U=0$, and
\item[(b)] The set of support $\tau$-rigid objects $\X$ in $\C(J(\U))$ such that $\delta(\X)=t$.
\end{itemize}
%The map $\E_{\U}$ induces a bijection between the set of support $\tau$-rigid objects
%$X$ in $\C(\Lambda)$ such that $X\amalg \U$ is support $\tau$-rigid and $\add X\cap \add %\U=0$, and the set of support $\tau$-rigid objects in $\C(J(\U))$.
\end{theorem}

From now on we assume $\Lambda$ is $\tau$-tilting finite.
Then, using Theorem \ref{rig-fin}, we obtain the following as a direct consequence of
Theorems~\ref{main-bi} and~\ref{main-biadd}.

\begin{corollary}\label{main-biW}
Assume $\Lambda$ is $\tau$-tilting finite, and let $\W$ be a wide subcategory of $\module\Lambda$. 
Let $\U$ be a support $\tau$-rigid object in $\C(\W)$.
Then there is a bijection $\E^{\W}_{\U}$ from
$$\{\X \in \ind(\C(\W)) \mid \X \amalg \U \text{ } \tau\text{-rigid}\}
\setminus \ind \U $$ 
to
$$\{\X \in  \ind(\C(J_{\W}(\U)) \mid \text{$\X$ is support $\tau$-rigid in $\C(J_{\W}(\U))$} \}.$$
Furthermore, the map $\E^{\W}_{\U}$ induces a bijection between:
\begin{itemize}
\item[(a)] The set of support $\tau$-rigid objects $\X$ in $\C(\W)$ such that $\X\amalg \U$ is support $\tau$-rigid, with
$\delta(\X) = t$ and $\add \X \cap \add \U=0$, and
\item[(b)] The set of support $\tau$-rigid objects $\X$ in $\C(J_{\W}(\U))$ with $\delta(\X) = t$.
\end{itemize}
\end{corollary}

%We prove Theorems~\ref{main-bi} and~\ref{main-biadd}
%in Section~\ref{bi}. 
Note that $[1]$ always denotes the shift in $D^b(\module\Lambda)$ rather than the shift in $D^b(\W)$ for some wide subcategory $\W$.

The next main ingredient is the following:

\begin{theorem}[Theorem \ref{main-comp}]\label{main-compo}
Assume $\Lambda$ is $\tau$-tilting finite. Let $\U$ and $\V$ be support $\tau$-rigid objects in $\C(\Lambda)$ with no common direct summands, and suppose that $\U\amalg \V$ is support $\tau$-rigid.
Then $\E_{\U}(\V)$ is support $\tau$-rigid in $\C(J(\U))$ and
the following equation holds:
$$J_{J(\U)}(\E_{\U}(\V)) = J(\U \amalg \V).$$
\end{theorem}

This has the following direct consequence, using Theorem \ref{rig-fin}.

\begin{corollary}\label{main-compoW}
Assume $\Lambda$ is $\tau$-tilting finite and let 
$\W$ be a wide subcategory of $\module\Lambda$. Let $\U$ and $\V$ be
support $\tau$-rigid objects in $\C(\W)$ with no common direct summands. Then $\E^{\W}_{\U}(\V)$ is support $\tau$-rigid in $\C(J_{\W}(\U))$ and
the following equation holds:
$$J_{J_{\W}(\U)}(\E^{\W}_{\U}(\V)) = J_{\W}(\U \amalg \V).$$
\end{corollary}

For a $\tau$-tilting finite algebra $\Lambda$, we can now define $\mathfrak{W}_{\Lambda}$ as follows.
The objects of $\mathfrak{W}_{\Lambda}$ are the wide subcategories of
$\module\Lambda$. Suppose $\W_1$ and $\W_2$ are two such wide
subcategories. If $\W_2\not\subseteq \W_1$, then we set $\Hom(\W_1,\W_2)=\emptyset$. Suppose that $\W_2\subseteq \W_1$.
Then we set 
%\begin{multline*}\Hom(\W_1,\W_2)=  \\ \{ g^{\W_1}_T\,:\,T \text{ is a basic support 
%$\tau$-rigid object
%in $\C(\W_1)$ such that $\W_2=J_{\W_1}(T)$}\},
%\end{multline*}
$$
\Hom(\W_1,\W_2)=  \left\{ g^{\W_1}_T\,\left| \ \ \begin{minipage}{0.5\textwidth} $T$ is a basic support 
$\tau$-rigid object in $\C(\W_1)$ \linebreak and $\W_2=J_{\W_1}(T)$ \end{minipage} \ \ \right. \right\},
$$
where $g_T^{\W_1}$ is a formal symbol associated to $\W_1$ and $T$.
Thus, in general $g_T^{\W}$ is a morphism in $\mathfrak{W}_{\Lambda}$ from
$\W$ to $J_{\W}(T)$.

Suppose that $\W_1$, $\W_2$ and $\W_3$ are wide subcategories of $\Lambda$ and $\W_3\subseteq \W_2\subseteq \W_1$.
Let $a\in \Hom(\W_1,\W_2)$ and $b\in\Hom(\W_2,\W_3)$. Then there are
support $\tau$-rigid objects $\U$ in $\W_1$ and $\overline{\V}$ in $\W_2$ such that $a=g^{\W_1}_{\U}$ and $b=g^{\W_2}_{\overline{\V}}$,
so that $\W_2=J_{\W_1}(\U)$ and $\W_3=J_{\W_2}(\overline{V})$.
By Theorem~\ref{main-biadd}, we can write
$\overline{\V}=\E^{\W_1}_{\U}(\V)$ for some support $\tau$-rigid object $\V$ in
$\C(\W_1)$ such that $\U\amalg \V$ is support$\tau$-rigid and $\add \V\cap \add \U=0$.
Thus, we have $b=g^{\W_2}_{\E^{\W_1}_{\U}(\V)}$.

By Theorem~\ref{main-compoW},
$$J_{W_1}(\U\amalg \V)=J_{J_{\W_1}(\U)}(\E_{\U}^{\W_1}(\V))=
J_{\W_2}(\E_{\U}^{\W_1}(\V))=J_{\W_2}(\overline{\V})=\W_3,$$
so we may define:
$$b\circ a=g^{\W_2}_{\E_{\U}(\V)} \circ g^{\W_1}_{\U} = g^{\W_1}_{\U \amalg \V},$$
since this is a morphism from $\W_1$ to $\W_3$.

For associativity of composition in $\mathfrak{W}_{\Lambda}$ we need the following theorem.

\begin{theorem}[Theorem~\ref{main-asso}]\label{main-as}
Assume $\Lambda$ is $\tau$-tilting finite, and 
let $\U$ and $\V$ be support $\tau$-rigid objects in $\C(\Lambda)$ with no common direct summands. Then $$\E_{\E_{\U}(\V)}^{J(\U)}\E_{\U} = \E_{\U \amalg \V}$$
\end{theorem}

The following is then a direct consequence, using Theorem \ref{rig-fin}.

\begin{corollary}\label{main-asW}
Assume $\Lambda$ is $\tau$-tilting finite, and 
let $\W$ be a wide subcategory of $\module\Lambda$. 
Let $\U$ and $\V$ be support $\tau$-rigid objects in $\C(\W)$ with no common direct summands, and suppose that $\U\amalg \V$ is support $\tau$-rigid in $\C(\W)$.
Then $$\E_{\E^{\W}_{\U}(\V)}^{J_{\W}(\U)}\E^{\W}_{\U} = \E^{\W}_{\U \amalg \V}$$
\end{corollary}

We are then in position to prove the following.

\begin{corollary}
The composition operation defined above is associative.
\end{corollary}

\begin{proof}
For a wide subcategory $\W$ of $\module \Lambda$ and support $\tau$-rigid object $\U$ in 
$\C(\W)$, let $\F^{\W}_{\U}$ denote the inverse of the bijection $\E^{\W}_{\U}$. 

Consider now maps $$\W_1 \xrightarrow{g^{\W_1}_{\U}} \W_2 \xrightarrow{g^{\W_2}_{\V}} \W_3 \xrightarrow{g^{\W_3}_{\W}} \W_4$$ where
$\W_2 = J_{\W_1}(\U)$, $\W_3 = J_{\W_2}(\V)$ and $\W_4 = J_{\W_3}(\W)$. Thus $\U$ is a support $\tau$-rigid object in $\C(\W_1)$,
the object $\V$ is support $\tau$-rigid in $\C(\W_2)$ and $\W$ is a support $\tau$-rigid object in $\C(\W_3)$,
and $\W_4\subseteq \W_3\subseteq \W_2\subseteq \W_1$.

We then have that $g^{\W_2}_{\V} \circ g^{\W_1}_{\U} = g^{\W_1}_{\U \amalg \F^{\W_1}_{\U}(\V)}$ and
$g^{\W_3}_{\W} \circ g^{\W_2}_{\V} = g^{\W_2}_{\V \amalg \F^{\W_2}_{\V}(\W)}$.  
Hence it follows that 
$$g^{\W_3}_{\W} \circ (g^{\W_2}_{\V} \circ g^{\W_1}_{\U}) = g^{\W_1}_{\U \amalg \F^{\W_1}_{\U}(\V) \amalg  \F^{\W_1}_{\U \amalg \F^{\W_1}_{\U}(\V)}(\W)}$$
and that  
$$(g^{\W_3}_{\W} \circ g^{\W_2}_{\V}) \circ g^{\W_1}_{\U} = g^{\W_1}_{\U \amalg \F^{\W_1}_{\U}(\V \amalg \F^{\W_2}_{\V}(\W))} =
g^{\W_1}_{\U \amalg \F^{\W_1}_{\U}(\V) \amalg \F^{\W_1}_{\U}\F^{\W_2}_{\V}(\W)}.$$

It follows from Theorem \ref{main-as} that $$\F^{\W_1}_{\U \amalg \F^{\W_1}_{\U}(\V)}
 = \F^{\W_1}_{\U}\F^{\W_2}_{\V}$$
and the claim follows. 
\end{proof}

Finally, we note that for each wide subcategory $\W$, we can consider
the trivial support $\tau$-rigid object $0$ in $\C(\W)$ which gives rise to a map
$g^{\W}_0:\W\rightarrow \W$. It easy to check that this satisfies the
axioms required for an identity map.
This completes the proof of the main result, Theorem~\ref{main}.

% there is the obvious identity
%map $\id_{\W} \colon \W \to \W$, and that we may consider
%the trivial $\tau$-rigid
%object $0$ in $\C(\W)$ such that $\id_{\W} = f_{0} \colon \W \to \W = J(0)$.

The paper is organized as follows. First, in Section \ref{not}, we give some 
background and notation. In Section \ref{bi} we prove Theorem \ref{main-bi} and Theorem \ref{main-biadd},
while in Section \ref{compo} we deal with Theorem \ref{main-compo}. Sections \ref{as} - 
\ref{s:mixed} are devoted
to the proof of Theorem \ref{main-as}.
In Section~\ref{s:irreducible}, we consider the morphisms in $\mathfrak{W}_\Lambda$ from a wide subcategory to a subcategory of corank 1, and in Section~\ref{s:factorization} we show how to interpret signed $\tau$-exceptional sequences in terms of factorizations of morphisms in $\mathfrak{W}_{\Lambda}$.
We conclude with an example in Section~\ref{examp}.

\section{Background and notation}\label{not}

Let $\P(\Lambda)$ denote the full
subcategory of projective objects in $\module \Lambda$ and
if $\mathsf{X}$ is a subcategory of $\module \Lambda$, let 
$\P(\mathsf{X})$ denote the full subcategory of $\mathsf{X}$ consisting of the Ext-projective
objects in $\mathsf{X}$, i.e. the objects $P$ in $\mathsf{X}$ such that $\Ext^1(P,X)=0$ for all $X\in \mathsf{X}$.

For an object $U$ in an additive category $\C$, let $\add U$ denote
the additive subcategory of $\C$ generated by $U$, i.e. the full subcategory of all direct summands in direct sums of copies of $U$. If $\mathsf{A}$ is abelian, we denote by $\Gen U$ the full subcategory of $\mathsf{A}$ consisting of all objects which are factor objects of objects in $\add U$.
We assume throughout that $\Lambda$ is basic and denote $\delta(\Lambda)$ by $n$.
We now recall notation and definitions of from \cite{air}. 

We consider $\module \Lambda$ as a full subcategory of $D^b(\module \Lambda)$ by regarding a module as a stalk complex concentrated in degree $0$, and we 
consider the full subcategory $\C(\Lambda) = \module \Lambda \amalg 
\module \Lambda [1]$ of $D^b(\module \Lambda)$.
For a module $M$, we denote by $\PP_M$ its minimal projective presentation,
considered as a two-term object in $K^b(\P(\Lambda)) \subseteq D^b(\module \Lambda)$.
Here, a two-term object in $\K$ is a complex of the form 
$$\cdots 0 \to 0 \to P^{-1} \to P^0 \to 0 \to 0 \to \cdots$$

The following summarizes some facts which we will use throughout the paper.

\begin{lemma}\label{rigid-rigid}
Let $U,X$ be in $\module \Lambda$.
\begin{itemize}
\item[(a)] \cite[Lemma 3.4]{air} $\Hom(U, \tau X)= 0$ if and 
only if $\Hom_{\D}(\PP_X, \PP_U[1]) = 0$. In particular, the module $U$ is 
$\tau$-rigid if and only if $\Hom_{\D}(\PP_U, \PP_U[1])= 0$.
\item[(b)] \cite[Theorem 5.10]{aus-sma} $\Hom(U, \tau X)= 0$ if and 
only if $\Ext^1(X, \Gen U) = 0$
\item[(c)] Let $\mathbb{X}$ and $\mathbb{Y}$ be two-term objects in $\K$.
Then $H^0$ induces an \sloppy  epimorphism 
$\Hom_{\K}(\mathbb{X},\mathbb{Y}) \to \Hom(H^0(\mathbb{X}), H^0(\mathbb{Y}))$
with kernel consisting of the maps factoring through $\add \Lambda[1]$.
\end{itemize}
\end{lemma}

We recall that if $U$ is a $\tau$-rigid module in $\module \Lambda$
then, by \cite[Theorem 5.8]{aus-sma} there is a torsion pair $(\Gen U,U^{\perp})$ in $\module\Lambda$.
We denote the corresponding torsion functors by $t_U:\module\Lambda\rightarrow \Gen U$ and $f_U:\module\Lambda\rightarrow U^{\perp}$.
If $U$ is $\tau$-rigid in $\W$, where $\W$ is a wide subcategory equivalent to a module category, we denote the corresponding torsion functors by $t^{\W}_U$
and $f^{\W}_U$ respectively.

\section{Bijection}\label{bi}
Let $\U$ be a an arbitrary (not necessarily indecomposable) 
support $\tau$-rigid object in $\C(\Lambda)$.
Then $\U = U \amalg P[1]$, where $U$ is a $\tau$-rigid module, $P$ is in $\P(\Lambda)$ and $\Hom(P,U) = 0$.
In this section, we will show that there is a bijection $\E_{\U}$ from the set 
$$\{\X \in \ind(\C(\Lambda)) \mid \X \amalg \U \text{ } \tau\text{-rigid}\} 
\setminus \ind \U$$ 
to $$\{\X \in \ind(\C(J(\U)) \mid \X \text{ } \tau\text{-rigid} \}$$

%In fact, we will Let $\W$ be a wide subcategory of $\module \Lambda$ which is 
%equivalent to a module category, and $\V$ a $\tau$-rigid object in $\C(\W)$.
%Then we will use the notation 
%$\E^{\W}_{\V}$ for the corresponding map. 

Such a map has already been defined in~\cite[Section 4-6]{bm} for the case
$\U$ is either a $\tau$-rigid module or a shift of a projective module,
so we first summarize the construction given there.

\begin{definition} \label{d:EUindecomposable}
Let $\U$ be a support $\tau$-rigid object in $\C(\Lambda)$ which is either
a module or a shift of a projective module. Suppose $\X$ lies in the set
$$\{\X \in \ind(\C(\Lambda)) \mid \X \amalg \U \text{ } \tau\text{-rigid}\} 
\setminus \ind \U.$$
Define $\E_{\U}(\X)$ in the following way.

\noindent \textbf{Case I}: $\U= U$ is a module.

\noindent Case I(a): If $X$ is in $\ind(\module \Lambda)$, $X\amalg U$ is
$\tau$-rigid and $X\not\in \Gen U$, then
$$\E_{U}(X) = f_U(X).$$

\bigskip

\noindent Case I(b): If $X$ is in $\ind(\module \Lambda)$ with $X \amalg U$ $\tau$-rigid and $X$ is in $\Gen U$, then 
$$\E_{U}(X) = f_U(H^0(\RR_X)[1]$$ 
where the triangle 
$$\RR_X \to \PP_{U_X} \to \PP_X \to $$
arises from the completion of the minimal right $\add \PP_{U}$-approximation
$\PP_{U_X} \to \PP_X$ to a triangle.
We have that $\RR_X = \PP_{B_X}$, for an indecomposable
direct summand $B_X$ of the Bongartz complement $B$ of $U$. Then we have 

\begin{itemize} 
\item[(i)] The triangle
\begin{equation}
\label{e:maintriangle}
\PP_{B_X} \to \PP_{U_X} \to \PP_X \to
\end{equation}
where the first map is a minimal left $\add \PP_{U}$-approximation and the second map is a minimal right $\add \PP_{U}$-approximation;
\item[(ii)] The exact sequence obtained from taking the homology of~\eqref{e:maintriangle}:
$$B_X \to U_X \to X \to 0,$$
where the first map is a minimal left $\add U$-approximation and the second
map is a minimal right $\add U$-approximation,
\end{itemize}
and we have $$\E_{U}(X) = f_U (B_X)[1].$$

The object $B_X$ is shown to be in $\P({^{\perp}\tau U})$ and $f_U (B_X)$ is in $\P(J(U))$, so 
$\E_{U}(X)$ is in $\ind \P(J(U))[1]$.

\bigskip

\noindent Case I(c): If $X$ is in $\ind (\P(\Lambda) \cap {^\perp{U}})[1]$, write $X=Q[1]$, with $Q$ in $\ind \P(\Lambda) \cap {^\perp{U}}$.
We have the triangle 
$$\PP_{B_X} \to \PP_{U_X} \to Q[1] \to $$
as in case (b), where the first map is a minimal left $\add \PP_{U}$-approximation and the second map is a minimal right $\add \PP_{U}$-approximation. Taking homology gives the exact sequence
$$Q \to B_X \to U_X \to 0$$
where the first map is a minimal left $\P({^\perp}(\tau U))$-approximation and the second map is a minimal left $\add U$-approximation.

We set
$$\E_{U}(X) = f_U (B_X)[1].$$

\noindent \textbf{Case II}: $\U= P[1]$, where $P$ is a projective module.

\bigskip

\noindent Case II(a): If $X$ is $\tau$-rigid in $\ind(\module \Lambda)$ and $\Hom(P,X) =0$, then set $\E_{P[1]}(X) = X$.

\bigskip

\noindent Case II(b): If $X= Q[1]$ with $Q$ in $(\ind (\P(\Lambda) 
\setminus \ind P)[1]$, then set
$\E_{P[1]}(X) = \E_{P[1]}(Q[1]) = f_P (Q)[1]$.

\end{definition}

\begin{theorem}\label{U-theorem} \cite[Proposition 5.6 and 5.10]{bm}
Let $\U$ be a support $\tau$-rigid object in $\C(\Lambda)$. Then we have the following.
\begin{itemize}
\item[(a)] If $\U=U$ is a module, then $\E_U$ gives a bijection between
\begin{itemize}
\item[(i)]
$$\{X \in \ind(\module \Lambda) \mid X \amalg U \text{ } \tau\text{-rigid, }
X \not \in \Gen U \}$$
and
$$\{X \in \ind(J(U)) \mid \text{$X$ $\tau$-rigid in $J(U)$} \}.$$
\item[(ii)]
 $$\{X \in  \ind(\module \Lambda) \mid  X \amalg U \text{ } \tau\text{-rigid, }
X \in  \Gen U \} \cup \{(\ind \P(\Lambda) \cap {^{\perp}U})[1] \}$$ and $$\{\ind 
\P(J(U))[1] \}.$$
\end{itemize}
\item[(b)] If $\U=P[1]$ is the shift of a projective module, then
$\E_{\U}$ gives a bijection between
$$(\{X \in \ind(\module \Lambda) \mid  X \text{ } \tau\text{-rigid}  \} \cap P^{\perp}) \cup 
(\ind \P(\Lambda)\setminus{\ind P})[1]$$
and
$$\{X \in \ind(J(\U)) \mid X \text{ } \tau\text{-rigid}  \} \cup  \ind(\P(J(\U))[1]$$
(noting that $J(\U)=P^{\perp}$ in this case).
\end{itemize}
\end{theorem}

We now consider the general case, where $\U= U \amalg P[1]$, for modules
$P,U$ with $P$ projective, is an arbitary support $\tau$-rigid object in
$\C(\Lambda)$. Note first that
$$\{\X \in \ind(\C(\Lambda)) \mid \X \amalg \U \text{ } \tau\text{-rigid}\} 
\setminus \ind \U$$
is the union of the sets
$$(\{X \in \ind(\module \Lambda) \mid X \amalg U \text{ } \tau\text{-rigid} \} \cap P^{\perp}) \setminus \ind U$$
and
$$((\ind \P(\Lambda) \cap {^\perp{U}})  \setminus \ind P)[1],$$
and that
\begin{multline*}
\{\X \in \ind(\C(J(\U)) \mid \text{$\X$ support $\tau$-rigid in $\C(J(U))$} \} \\ = \{X \in \ind(J(\U)) \mid \text{$X$ $\tau$-rigid in $J(U)$} \} \cup  \ind(\P(J(\U))[1],
\end{multline*}
so we next analyse the behaviour of $\E_{U}$ when applied to a module
$X\in P^{\perp}$.

\begin{lemma}\label{rest-Lemma}
Let $U$ be a $\tau$-rigid module. Then:
\begin{itemize}
\item[(a)] The map $\E_{U}$ restricts to a bijection between
$$\{X \in  \ind(\module \Lambda) \mid X \amalg U \text{ } \tau\text{-rigid, }
X \not \in \Gen U 
 \} \cap P^{\perp}$$ and
$$\{X \in \ind(J(U)) \mid \text{$X$ $\tau$-rigid in $J(U)$} \}
\cap P^{\perp}$$ 
\item[(b)]The map $\E_U$ restricts to a bijection between 
$$
 \{X \in \ind(\module \Lambda) \mid X \amalg U \text{ } \tau\text{-rigid, }
X \in  \Gen U \} \cap P^{\perp} \\ \cup 
((\ind \P(\Lambda)  \setminus \ind P) \cap {^\perp U})[1]   
$$
and $$\ind(\P(J(U)))[1]  \setminus \ind \E_U(\add P[1]).$$
\item[(c)] The map $\E_{U}$ restricts to a bijection between
$$(\{X \in  \ind(\module \Lambda) \mid  X \amalg U \text{ } \tau\text{-rigid }
 \} \cap P^{\perp})\cup   ((\ind \P(\Lambda) \setminus \ind P) \cap {^\perp U}))[1]  $$  

and
$$(\{X \in  \ind(J(U)) \mid \text{$X$ $\tau$-rigid in $J(U)$} \} \cap P^{\perp} ) \cup
 (\ind \P(J(U))[1] \setminus \ind \E_U(P[1])
$$ 
\end{itemize}
\end{lemma}

\begin{proof}
\noindent (a) Let $X$ be in $\ind(\module \Lambda)$ with $X \amalg U$ a $\tau$-rigid module, and such that
$X$ is not in $\Gen U$. Then $\E_U(X) = f_U(X)$.
Since $\Hom(P,U) = 0$, we also have $\Hom(P,\Gen U) = 0$ and in particular $\Hom(P,t_U(X)) = 0$. Since $P$ is projective, it then follows that 
$\Hom(P,X) \simeq \Hom(P, f_U (X)) =
\Hom(P, \E_U(X))$. 
Hence the claim follows, using 
Theorem~\ref{U-theorem}(a).

\noindent (b) Since $\Hom(P,U) = 0$, we have $\Hom(P, \Gen U) = 0$, and the 
claim follows \sloppy Theorem~\ref{U-theorem}(b).

\noindent (c) The claim follows directly from combining (a) and (b).
\end{proof}

If $\W$ is a wide subcategory of $\module\Lambda$ which is equivalent
to a module category, and $\U$ is a support $\tau$-rigid object in $\C(\W)$ which is either a module or the shift of a projective object in $\W$,
then we denote by $\E^{\W}_{\U}$ the map corresponding to that defined in Definition~\ref{d:EUindecomposable}.

Note that $\E_U(P[1]) = P[1]$, so we have the map
$$\E_{P[1]}^{J(U)} = \E_{\E_U(P[1])}^{J(U)}.$$

\begin{lemma}\label{comp}
Let $U$ be a $\tau$-rigid module.
Then the set $$(\{X \in  \ind(J(U)) \mid \text{$X$ $\tau$-rigid in $J(U)$} \} \cap P^{\perp} ) \cup
 (\ind \P(J(U))[1] \setminus \ind \E_U(P[1])
$$ 
is the domain of $\E^{J(U)}_{\E_U(P[1])}$.
\end{lemma}

\begin{proof}
Let $Q$ in $\P(J(U))$ be such that $\E_U(P[1]) = Q[1]$.
Recall (see Definition~\ref{d:EUindecomposable}, Case I(c))
that $Q = f_U(Y_P)$,
where $P \to Y_P$ is a minimal left ${^\perp (\tau U)}$-approximation,
and there is an exact sequence
\begin{equation}\label{seq1}
P \to Y_P \to U_P \to 0
\end{equation}
with $U_P$ in $\add U$.

We claim that 
\begin{equation}\label{P-Q}
J(U) \cap P^{\perp} = J(U) \cap Q^{\perp}.
\end{equation}
It is clear by the definition of $\E^{J(U)}_{\E_U(P[1])}$ that the 
assertion of the lemma follows from this claim.

In order to prove the claim, let $M$ be in $J(U) \cap P^{\perp}$ and apply 
the right exact functor $\Hom(\ ,M)$ to the sequence \eqref{seq1}, to obtain the exact sequence
$$0 \to \Hom(U_P, M) \to \Hom(Y_P,M) \to \Hom(P,M).$$
 We have by assumption that $\Hom(U_P, M) =0 = \Hom(P,M)$, and hence \sloppy also  
$\Hom(Y_P,M) =0$. It then follows that $\Hom(Q,M) = 0$, since there is
an epimorphism $Y_P \to f_U (Y_P) = Q$. So we have $J(U) \cap P^{\perp} \subseteq 
J(U) \cap Q^{\perp}$.

Conversely, suppose $M$ is in $J(U) \cap Q^{\perp}$. Consider the canonical sequence
$$0 \to t_U(Y_P) \to Y_P \to f_U (Y_P) (= Q) \to 0$$
for $Y_P$, and apply $\Hom(\ , M)$ to obtain the exact sequence
$$0\to \Hom(Q,M) \to \Hom(Y_P,M) \to \Hom(t_U (Y_P),M)$$
We have by assumption $\Hom(Q,M) = 0$, and $\Hom(t_U (Y_P),M) = 0$, since $t_U (Y_P)$ is in 
$\Gen U$ and $M$ is in $U^{\perp}$. Hence, also $\Hom(Y_P,M) = 0$.
By Lemma \ref{rigid-rigid}, we then have 
\begin{equation}\label{eq-1}
\Hom(\PP_{Y_P},\PP_M)/(\add \Lambda[1]) = 0.
\end{equation}
We have the following triangle (from the computation of $\E_U(P[1]) = Q[1]$;
see Definition~\ref{d:EUindecomposable}, Case I(c)).
\begin{equation}\label{tria-1}
\PP_{U_P}[-1] \to P \to \PP_{Y_P} \to \PP_{U_P}
\end{equation}
Now let $\alpha \colon P \to \PP_M$ be arbitrary. Since $M$ is in $J(U)$, we have
$\Hom(M, \tau U)$= 0 and so by Lemma \ref{rigid-rigid}, we have $\Hom(\PP_{U_P}, \PP_M[1]) = 0$.
Applying $\Hom(\ , \PP_M)$ to the triangle~\eqref{tria-1}, we obtain that 
$\Hom(\PP_{Y_P}, \PP_M) \to \Hom(P,\PP_M)$ is surjective, and hence that $\alpha$
factors through a map $\PP_{Y_P} \to \PP_M$ and hence through $\Lambda[1]$ by
\eqref{eq-1}.
We have $\Hom(P,\Lambda[1]) = 0$ and hence we obtain $\Hom(P, \PP_M) = 0$.
So we have $J(U) \cap Q^{\perp} \subseteq 
J(U) \cap P^{\perp}$, and this finishes the proof of the claim that 
$J(U) \cap Q^{\perp} =
J(U) \cap P^{\perp}$, and hence the proof of the lemma.
\end{proof}

By Lemmas~\ref{rest-Lemma}(c) and~\ref{comp}, the composition
$\E_{\E_U(P[1])}^{J(U)} \E_U$ is a well-defined map
with domain $$\{X \in  \ind(\C(\Lambda)) \mid  X \amalg \U \text{ } \tau\text{-rigid}\} \setminus \ind \U.$$ 
We make the following definition:

\begin{definition}
Let $U$ and $P$ be modules such that
$\U = U \amalg P[1]$ is a support $\tau$-rigid object in $\C(\Lambda)$.
We set $\E_{\U} \colon = \E_{\E_U(P[1])}^{J(U)} \E_U$.
\end{definition}

We can now prove the main result of this section.

\begin{theorem}\label{thm-bi}
Let $\U = U \amalg P[1]$ be a support $\tau$-rigid object in $\C(\Lambda)$.
Then the map $\E_{\U}$ is a bijection between the sets
$$\{\X \in  \ind(\C(\Lambda)) \mid  \X \amalg \U \text{ } \tau\text{-rigid}\}
\setminus \ind \U$$ 
and $$\{\X \in \ind(\C(J(\U))) \mid \text{$\X$ support $\tau$-rigid in $\C(J(\U))$} \}.$$
\end{theorem}

\begin{proof}
First note that if $P = 0$ or $U=0$, this is proved in \cite[Proposition 5.6 and 5.10]{bm}.

Using Lemma \ref{rest-Lemma}(c) and \eqref{P-Q} and the fact that
$\E_{U}(P[1]) = Q[1]$, we have that $\E_U$ restricts to a bijection
between
$$(\{X \in  \ind(\module \Lambda) \mid  X \amalg U \text{ } \tau\text{-rigid}
 \} \cap P^{\perp}) \cup   ((\ind \P(\Lambda) \setminus \ind P) \cap {^\perp U})[1]    $$  
and
$$(\{X \in \ind(J(U)) \mid X \text{ } \tau\text{-rigid}  \} \cap Q^{\perp})  \cup
 (\ind \P(J(U)) \setminus \ind Q)[1] \}.
$$ 

The target of this map is the domain of $\E_{Q[1]}^{J(U)} = \E_{\E_U(P[1])}^{J(U)}$.
Moreover (see Case II in Definition~\ref{d:EUindecomposable}), the map
$\E_{Q[1]}^{J(U)}$ gives a bijection between
%\begin{multline*} 
$$
\{\{X \in  \ind(J(U)) \mid X \text{ } \tau\text{-rigid} \} \cap Q^{\perp} \} \cup  (\ind \P(J(U))) \setminus \ind Q[1] 
$$
%\end{multline*}
and
$$\{X \in  \ind(J(U)) \mid X \text{ } \tau\text{-rigid} \} \cup  \ind(\P(J(\U))[1].$$

This finishes the proof of the claim.
\end{proof}

Note that we have so far only defined $\E_{\U}(\X)$ for an object $\X$ in the set
$$\{\X \in \ind(\C(\Lambda)) \mid \X \amalg \U \text{ } \tau\text{-rigid}\} 
\setminus \ind \U.$$
However, we will also need to consider $\E_{\U}$ as a map from the set of
all basic objects $\X$ (not necessarily indecomposable) in $\C(\Lambda)$
such that $\X \amalg \U$ is support $\tau$-rigid and $\add \X\cap \add \U=0$,
to the set of all support $\tau$-rigid objects in $\C(J(\U))$.
So for such $\X = \X_1 \amalg \cdots \amalg \X_t$ where the $\X_i$ are 
indecomposable, we define
$\E_{\U}(\X) = \E_{\U}(\X_1) \amalg \cdots \amalg \E_{\U}(\X_t)$.
 
\begin{theorem}\label{rigid-sums}
Let $\U$ be a support $\tau$-rigid object in $\C(\Lambda)$ with $\delta(\U) = t'$.
For any positive integer $t \leq n- t'$, 
the map $\E_{\U}$ induces a bijection between the set of basic support $\tau$-rigid objects $\X$ in $\C(\Lambda)$ such that
$\delta(X) = t$, with $\X \amalg \U$ support $\tau$-rigid and $\add \X \cap \add \U=0$, and the set of basic support $\tau$-rigid objects $\Y$ in $\C(J(\U))$ with $\delta(\Y)=t$.
\end{theorem} 
\begin{proof} Recall that by definition $\E_{\U} = \E_{\E_U(P[1])}^{J(U)} \E_U$,
so the result follows from~\cite[Prop. 6.7, Prop. 6.10]{bm}.
\end{proof}

%\begin{proposition}\label{rigid-sums}
%Let $\U$ be $\tau$-rigid in $\C(\Lambda)$.
%\begin{itemize}
%\item[(a)] If $X$ in $\C(\Lambda)$ is such that $\U \amalg X$ is $\tau$-rigid, 
%and $\add X\cap \add U=0$, then
%$\E_{\U}(X)$ is $\tau$-rigid in $\C(\U)$. 
%\item[(b)] If $Y$ is basic $\tau$-rigid in $J(\U)$ with $\delta(Y) =t$, there exists a 
%unique basic
%$X$ in $\C(\Lambda)$, such that $\U \amalg X$ is $\tau$-rigid,
%and $\E_{\U}(X) =Y$ with $\add \U \cap \add X = 0$,
%and we have $\delta(X) = t$.
%\end{itemize} 
%\end{proposition} 

%\begin{proof} Recall that by definition $\E_{\U} = \E_{\E_U(P[1])}^{J(U)} \E_U$.
%\noindent (a) By \cite[?]{bm}, we have that $\E_U(X)$ is $\tau$-rigid in $J(U)$, and 
%by
%\cite[?]{bm}, we then have that $\E_{\E_U(P[1])}^{J(U)} \E_U(X)$ is $\tau$-rigid in
%$J(\U)$.
%\noindent (b) This follows from \cite[]{bm}.
%\end{proof}

\begin{lemma}\label{tau-in-pperp}
Let $U$ be a $\tau$-rigid module, and $P$ a projective module with $\Hom(P,U) = 0$.
Then
\begin{equation*}
 ^{\perp}(\tau_{P^{\perp}}U)  \cap P^{\perp} = 
  \cap {^{\perp}(\tau U)} \cap P^{\perp}. 
\end{equation*}
\end{lemma}

\begin{proof}
We have 
\begin{align}
\label{eq-x} {^{\perp}(\tau_{P^{\perp}} U}) \cap P^{\perp}  & =
  \{ Y \in \module \Lambda \mid \Ext^1(U, \Gen_{P^{\perp}} Y)= 0\} \cap P^{\perp}  \\ \label{eq-y}
& =   \{ Y \in \module \Lambda \mid \Ext^1(U, \Gen Y)= 0\} \cap P^{\perp} \\ \nonumber
& = {^{\perp}(\tau U)} \cap P^{\perp},
\end{align}
where \eqref{eq-x} holds by Lemma \ref{rigid-rigid}, and  \eqref{eq-y} holds since
$\Gen_{P^{\perp}} Y = \Gen Y$ for $Y$ in $P^{\perp}$.
\end{proof}

\section{Composition}\label{compo}

The aim of this section is to prove Theorem \ref{main-compo}.

If $\A$ is (a category equivalent to) a module category, we let $r(\A)$ \sloppy denote the rank of the 
Grothendieck group of $\A$, that is: the number of simple objects in 
$\A$ up to isomorphism. Recall that $\delta(X)$ denotes the number of indecomposable summands in a basic object $X$. We always write
$r(\module \Lambda) = n$. Recall the following important facts.
 
\begin{proposition}\label{prop-wide}
Let $\U$ be a $\tau$-rigid object in $\module \Lambda$. Then the following hold.
\begin{itemize}
\item[(a)] \cite[Theorem 3.28]{dirrt} $J(\U)$ is a wide subcategory of $\module \Lambda$. 
\item[(b)] \cite[Theorem 3.8]{jasso} $J(\U)$ is equivalent to a module category with rank $r(J(\U)) = 
n- \delta(\U)$.
\end{itemize}
\end{proposition}

The following results are crucial.
\begin{proposition}\label{tau-finite}
Assume that $\Lambda$ is $\tau$-tilting finite. 
\begin{itemize}
\item[(a)] For each wide subcategory $\W$ of $\module \Lambda$, there is a 
support $\tau$-rigid object $\U$ in $\C(\Lambda)$ such that $\W = J(\U)$.
\item[(b)] If $\Lambda$ is $\tau$-tilting finite, then each wide subcategory $\W$ of $\module \Lambda$ is $\tau$-tilting finite.
\end{itemize}
\end{proposition}

\begin{proof}
\noindent (a) This is contained in Theorem 3.34 in \cite{dirrt}.

\noindent (b) This is a direct consequence of Theorem \ref{thm-bi}, using (a). 
\end{proof}

We will from now on assume $\Lambda$ is a $\tau$-tilting finite algebra.

In this section we prove the following (Theorem~\ref{main-compo}).
\begin{theorem}\label{main-comp}
Let $\U$ and $\V$ be objects in $\C(\Lambda)$ with no common direct summands,
and suppose that $\U\amalg \V$ is support $\tau$-rigid. Then $\E_{\U}(\V)$ is support $\tau$-rigid in $\C(J(\U))$ and
the following equation holds
$$J_{J(\U)}(\E_{\U}(\V)) = J(\U \amalg \V).$$
\end{theorem}

This theorem is the key for proving that composition is well-defined in 
the category $\mathfrak{W}_{\Lambda}$. Note first that by Theorem \ref{rigid-sums}, 
we have 
that $\E_{\U}(\V)$ is support $\tau$-rigid in $\C(J(\U))$. The remainder of this section is 
devoted to proving the second assertion of the Theorem.

We first make the following observation.

\begin{lemma}\label{rank-lemma}
In the setting of Theorem \ref{main-comp} we have
$$r(J_{J(\U)}(\E_{\U}(\V))) = r(J(\U \amalg \V)).$$
\end{lemma}

\begin{proof}
Let $r(\module \Lambda)= n$.
By \cite{jasso} we have 
$r(J(\T)) = n - \delta(\T)$ for any support $\tau$-rigid object $\T$ in $\C(\Lambda)$.
So $r(J(\U) = n- \delta(\U)$ and
$r(J(\U \amalg \V)) = n- \delta(\U) -\delta(\V)$. Furthermore 
$r(J_{J(\U)}(\E_{\U}(\V))) = (n - \delta(\U))) - (\delta(\E_{\U}(\V))) = 
n- \delta(\U) -\delta(\V)$, and the claim follows.
\end{proof}

\begin{lemma}\label{abc-lemma}
Let $\A$ be an abelian category and $\A'' \subseteq \A'$ wide subcategories of $\A$.
Then $\A''$ is a wide subcategory of $\A'$.
\end{lemma}

\begin{proof}
This follows directly from the fact that a subcategory is
wide if and only if it is closed under kernels, cokernels and
extensions.  
\end{proof}

\begin{proof}[Proof of Theorem \ref{main-comp}]
We first claim it is sufficient to prove 
$$ J(\U \amalg \V) \subseteq J_{J(\U)}(\E_{\U}(\V)).$$
If this holds then, by Lemma \ref{abc-lemma}, we have that $J(\U \amalg \V)$ is a wide subcategory of
$J_{J(\U)}(\E_{\U}(\V))$. Then, by Proposition \ref{tau-finite}, there is a support $\tau$-rigid object 
$\V'$ in $\C(J_{J(\U)}(\E_{\U}(\V)))$ such that 
$$J(\U \amalg \V) = J_{J_{J(\U)}(\E_{\U}(\V))}(\V')$$
We have $r(J_{J_{J(\U)}(\E_{\U}(\V))}(\V')) = n- \delta(\U) - \delta(\V) - \delta(\V')$
by Proposition~\ref{prop-wide}(b) and Theorem~\ref{rigid-sums}.
Hence $r(\V') = 0$, so $\V'= 0$, and we have
$$J(\U \amalg \V) = J_{J(\U)}(\E_{\U}(\V)).$$  

In order to prove 
\begin{equation}\label{inc} 
J(\U \amalg \V) \subseteq J_{J(\U)}(\E_{\U}(\V))
\end{equation}
we first discuss various special cases.

\bigskip 

\noindent {\bf Case I:}
Let $U$ be $\tau$-rigid in $\module \Lambda$, and $V \not \in\Gen U$, such that
$\overline{V}=\E_U(V)$ is $\tau$-rigid in $J(U)$. Then $\overline{V} = f_U(V)$, and there
is an epimorphism $V \to \overline{V}$.

Let $M$ be in $J(U \amalg V)$. Then we have $M \in J(V) \subseteq V^{\perp}$, and since
$0 \to \Hom(\overline{V},M) \to \Hom(V,M)$ is exact, we also have
$\Hom(\overline{V},M) =0$.

We next need to show $\Hom(M, \tau_{J(U)}\overline{V}) = 0$.
By Lemma \ref{rigid-rigid}, this is equivalent to showing 
$\Ext^1(\overline{V}, \Gen_{J(U)} M)= 0$. We have $\Gen_{J(U)} M =
\Gen M \cap J(U)$, and hence it is sufficient to prove 
$\Ext^1(\overline{V}, \Gen M \cap J(U))= 0$. 
Let $M'$ be in $\Gen M \cap J(U)$. Apply $\Hom(\ ,M')$ to the canonical sequence 
$0 \to t_U(V) \to V \to f_U (V) = \overline{V} \to 0$
for $V$, to obtain the exact sequence 
\begin{equation}\label{e-seq}
\Hom(t_U(V),M') \to \Ext^1(\overline{V},M')  \to \Ext^1(V,M').
\end{equation}
The first term in \eqref{e-seq} vanishes, since $t_U (V)$ is in $\Gen U$ and $M'$ is in $U^{\perp}$. We have $\Hom(M, \tau V) = 0$, since $M$ is in $J(V)$, so $\Hom(M', \tau V)= 0$, since $M'$ is in $\Gen M$. Using the AR-formula, we obtain that the third term in \eqref{e-seq} also vanishes, and hence also the second term vanishes.
Hence we have that $\Ext^1(\overline{V}, \Gen M \cap J(U)) = 0$ and so
$\Hom(M, \tau_{J(U)} \overline{V}) = 0$. So $M$ is in $J_{J(U)}(\overline{V})$,
and we have shown inclusion \eqref{inc} in this case.
 
\bigskip
 
\noindent {\bf Case II (a):}
Let $U$ be $\tau$-rigid in $\module \Lambda$, and $V$ in $\Gen U$ such that
$\E_{U}(V) = \overline{V}$ is in $\P(J(U))[1]$. Recall that $\overline{V}$ is
computed as follows.
We have a triangle $$\xymatrix{
\PP_{B_V} \ar[r]^a &  \PP_{U_V}  \ar[r]^b &  \PP_{V} \ar[r]^(0.4){c} & \PP_{B_V}[1]}$$
where $a$ is a minimal left $\add \PP_{U}$-approximation and $b$ is a minimal right $\add \PP_{U}$-approximation, and taking homology gives
the exact sequence
$$\xymatrix{
B_V \ar[r]^{a'} &  U_V \ar[r]^{b'} &  V \ar[r] & 0}$$  
where $a'$ is a minimal left $\add U$-approximation and $b'$
is a minimal right $\add U$-approximation. Let $\overline{Q} = f_U (B_V) $.
Then $\overline{V} = \overline{Q}[1]$.

Now suppose that $M$ lies in $J(U \amalg V)$. Note that $J_{J(U)}(\E_U(V)) = J(U) \cap 
\overline{Q}^{\perp}$. Since $M$ is in $J(U)$, it is sufficient to
show that $\Hom(\overline{Q},M) = 0$. Since $\overline{Q}$ is a quotient of
$B_V$, it is sufficient to show that $\Hom(B_V,M) = 0$. For this
let $g:\PP_{B_V} \to \PP_{M}$ be an arbitrary map. By Lemma \ref{rigid-rigid}, we have that
$\Hom(\PP_V, \PP_{M}[1]) =0$, since $\Hom(M,\tau V) = 0$. 
Hence, the composition $g\circ c[-1]:\PP_V[-1] \to \PP_{M}$ vanishes, and there is a factorization $g=ha$ for some $h:\PP_{U_V}\rightarrow \PP_M$:
$$\xymatrix{
\PP_V[-1] \ar[r]^(0.55){c[-1]} & \PP_{B_V} \ar[d]_g \ar[r]^a &  \PP_{U_V} \ar[dl]^h \ar[r]^b &  \PP_{V} \ar[r]^(0.4){c} & \PP_{B_V}[1] \\
& \PP_M}$$
Since $\Hom(U,M) = 0$, we have $\Hom(\PP_{U_V}, \PP_M)/\Lambda[1] = 0$, and
it follows that $\Hom(\PP_{B_V}, \PP_M)/\Lambda[1] = 0$, and hence by
Lemma \ref{rigid-rigid}, we have $\Hom(B_V,M)= 0$. Hence we have shown inclusion \eqref{inc} in this case.

\bigskip

\noindent {\bf Case II (b):} 
Let $U$ be $\tau$-rigid in $\module \Lambda$, and $V \in (\P(\Lambda) \cap 
{^{\perp}U})[1]$. Assume $V= Q[1]$ for $Q$ in $\P(\Lambda) \cap 
{^{\perp}U}$.

Recall that $\E_{U}(V) = \overline{V}$ is computed as follows. 
There is an exact sequence
$$Q \to B_V \to U_V \to 0$$
where the first map is a minimal left ${^\perp(\tau U)}$-approximation
(or, equivalently, a minimal left $\P({^\perp(\tau U)})$-approximation), and
$\overline{V} = f_U(B_V)[1]$; we set $\overline{Q}=f_U(B_V)$.

Now let $M$ be in $J(U \amalg V)$, that is $M$ is in $J(U)$ and $\Hom(Q,M) = 0$.
We need to prove that $\Hom(\overline{Q},M) = 0$. Since $\overline{Q} = f_U(B_V)$
is a quotient of $B_V$, it is sufficient to show that 
$\Hom(B_V,M) = 0$. Recall that there is a triangle
$$Q \to \PP_{B_V} \to \PP_{U_V} \to$$ 
and consider
an arbitrary map $\PP_{B_V} \to \PP_M$.
The composition $Q \to \PP_{B_V} \to \PP_M$ vanishes, since $\Hom(Q,M) = 0$
and hence $\Hom(Q, \PP_M) = 0$, by Lemma \ref{rigid-rigid}.
Therefore, the map  $\PP_{B_V} \to \PP_M$ factors  $\PP_{B_V} \to \PP_{U_V} \to \PP_M$.
Since $M$ is in $J(U) \subseteq U^{\perp}$, we have $\Hom(U_V,M) =0$, so 
$\Hom(\PP_{U_V}, \PP_{M})/\add \Lambda [1] = 0$. Hence also 
$\Hom(\PP_{B_V}, \PP_{M})/\add \Lambda [1] = 0$ and $\Hom(B_V,M)= 0$ as required.
Hence we have shown that the inclusion \eqref{inc} holds also in this case.

\bigskip

\noindent {\bf Case III:}
Let $U=P[1]$ with $P$ in $\P(\Lambda)$, and let $V$ be $\tau$-rigid. Then
$J(U) = P^{\perp}$ and
$\E_U(V) = \overline{V} = V$ is also $\tau_{J(U)}$-rigid, 
 by \cite[Lemma 2.1]{air}. Furthermore, by Lemma \ref{tau-in-pperp}  
we have 
\begin{equation*}
J_{P^{\perp}}(V) = P^{\perp} \cap V^{\perp} \cap {}^{\perp}(\tau_{P^{\perp}}V) = 
 P^{\perp} \cap V^{\perp} \cap {}^{\perp}(\tau V) = J(U \amalg V),
\end{equation*}
which finishes the proof of case III.

\bigskip

\noindent {\bf Case IV:}
Now let $U=P[1]$ and $V=Q[1]$, for $P,Q \in \P(\Lambda)$.
Then $\E_U(V)= \overline{V} = (f_P Q)[1]$. For an object $M$ in $P^{\perp}$,
apply $\Hom(\ , M)$ to the exact sequence
$$0 \to t_P(Q) \to Q \to f_P (Q) \to 0$$
to obtain the exact sequence
$$0 \to \Hom(f_P (Q), M) \to \Hom(Q,M) \to \Hom(t_P (Q), M)$$
The last term vanishes, since $t_P (Q)$ is in $\Gen P$, so 
$\Hom(f_P (Q), M) \simeq \Hom(Q,M)$. Hence, we have
$J_{J(U)}(\E_U(V))= J_{P^{\perp}}(\overline{V})  = P^{\perp} \cap (f_P Q)^{\perp}  
= P^{\perp} \cap Q^{\perp} = J(U\amalg V)$,
which finishes the proof of case IV.

\bigskip

\noindent {\bf General case.}
Let $\U = U \amalg P[1]$ and $\V= V \amalg Q[1]$, for $U,V$ $\tau$-rigid modules
and $P,Q$ in $\P(\Lambda)$. We assume that $\U\amalg \V$ is support $\tau$-rigid in
$\C(\Lambda)$.  We proceed by induction on the rank $n = r(\module \Lambda)$. 
We therefore first assume $U \neq 0$, so $r(J(U)) < n$.

Then
\begin{equation}\label{eq-first}
J(\U \amalg \V) = J(U \amalg V) \cap P^{\perp} \cap Q^{\perp}
\end{equation}
and
\begin{align}\nonumber
J_{J(\U)}(\E_{\U}(\V))  & = J_{J(U) \cap P^{\perp}}(\E_{U \amalg P[1]}(V \amalg Q[1]))
 \\\label{eq-a}
 & = J_{J(U) \cap P^{\perp}}(\E_{\E_U(P[1])}^{J(U)}\E_U(V \amalg Q[1])) \\ \label{eq-b}
  & = J_{J(U) \cap P^{\perp}}(\E_{\E_U(P[1])}^{J(U)}\E_U(V)) \cap
    J_{J(U) \cap P^{\perp}}(\E_{\E_U(P[1])}^{J(U)}\E_U(Q[1]))
\end{align}

Where (\ref{eq-a}) is by definition by of $\E_{\U} = \E_{U\amalg P[1]}$.

Note that we have 
\begin{equation}\label{eq-third}
J(U) \cap P^{\perp} = J(U \amalg P[1]) = J_{J(U)}(\E_U(P[1]))
\end{equation}
by case II(b).

We next compute the terms of (\ref{eq-b}) separately. For the first term, we obtain

\begin{align}
\label{eq-c} J_{J(U) \cap P^{\perp}}(\E_{\E_U(P[1])}^{J(U)}\E_U(V)) & =
J_{J_{J(U)}(\E_U(P[1]))}(\E_{\E_U(P[1])}^{J(U)}\E_U(V))  \\ \label{eq-d}
& = J_{J(U)}(\E_U(P[1]) \amalg \E_U(V))  \\ \nonumber
& = J_{J(U)}(\E_U(P[1])) \cap  J_{J(U)}(\E_U(V)) \\ \label{eq-e}
& = J(U \amalg P[1]) \cap  J_{J(U)}(\E_U(V))  \\ \nonumber
& = J(U) \cap P^{\perp} \cap  J_{J(U)}(\E_U(V))  \\ \label{eq-f}
& = P^{\perp} \cap  J_{J(U)}(\E_U(V)) 
\end{align}
where (\ref{eq-c}) follows from (\ref{eq-third}), and (\ref{eq-d}) is obtained by using the
induction assumption for the proper subcategory $J(U)$, while (\ref{eq-e}) holds by
case II(b) and (\ref{eq-f}) holds by $ J_{J(U)}(\E_U(V)) \subseteq J(U)$.

Similarly, for the second term in \eqref{eq-b}, we obtain
\begin{align}
\nonumber 
J_{J(U) \cap P^{\perp}}(\E_{\E_U(P[1])}^{J(U)}\E_U(Q[1]))  & = 
J_{J_{J(U)}(\E_U(P[1]))}(\E_{\E_U(P[1])}^{J(U)}\E_U(Q[1]))  \\ \nonumber
& = J_{J(U)}(\E_U(P[1]) \amalg \E_U(Q[1]))  \\ \nonumber
& = J_{J(U)}(\E_U(P[1] \amalg Q[1]))  \\ \nonumber
& = J(U \amalg P[1] \amalg Q[1])  \\ \label{eq-bb}
& = J(U) \cap P^{\perp} \cap Q^{\perp} 
\end{align}

We then obtain
\begin{align}
\label{likn1} J_{J(\U)}(\E_{\U}(\V)) 
  & = J_{J(U) \cap P^{\perp}}(\E_{\E_U(P[1])}^{J(U)}\E_U(V)) \cap
    J_{J(U) \cap P^{\perp}}(\E_{\E_U(P[1])}^{J(U)}\E_U(Q[1])) \\ \label{likn2}
    & = J_{J(U)}(\E_U(V)) \cap P^{\perp} \cap J(U) \cap P^{\perp} \cap Q^{\perp} \\ \label{likn3}
 &= J(U\amalg V)\cap P^{\perp} \cap Q^{\perp} \\ \label{likn4}
 &   =J(\U \amalg \V)
\end{align}
where (\ref{likn1}) is (\ref{eq-b}) and where (\ref{likn2}) follows from combining  
 (\ref{eq-f}) and (\ref{eq-bb}). Furthermore (\ref{likn3}) follows from 
Cases I and II(a) and (\ref{likn4}) follows from \eqref{eq-first} respectively.

So we have that the claim of the theorem holds in the general case, with the assumption that $U \neq 0$.

Now, consider the case where $U = 0$.

We then have 
\begin{align}
\nonumber J(\U \amalg \V) & = J(P[1] \amalg V \amalg Q[1]) \\ \label{eq-g}
& = J(V) \cap P^{\perp} \cap Q^{\perp}
\end{align}

and
\begin{align}\nonumber
J_{J(\U)}(\E_{\U}(\V))  & = J_{J(P[1])}(\E_{P[1]}(V \amalg Q[1]))
 \\ \nonumber
 & = J_{J(P[1])}(\E_{P[1]}(V)) \cap J_{J(P[1])}(\E_{P[1]}(Q[1])) \\ \label{eq-h}
  & = J(V \amalg P[1]) \cap J(P[1] \amalg Q[1]) \\ \nonumber
  & = J(V) \cap P^{\perp}  \cap Q^{\perp} \\ \nonumber
  & = J(\U \amalg \V)
\end{align}
where for (\ref{eq-h}), we use cases III and IV. This finishes the proof for the case $U = 0$, and hence the proof of the theorem.
\end{proof}

\section{Associativity}\label{as}

The aim of this section is to prove that the composition operation defined in Section \ref{compo}
is associative. The main step is to prove Theorem \ref{main-as}. We prepare for this, by giving several
useful lemmas.

\begin{lemma}\label{al1}
Let $U,X,Y$ be in $\module \Lambda$ where $U$ is $\tau$-rigid and $\Hom(U,\tau X)= 0$.
Then the induced map $\alpha \colon \Hom(X,Y) \to \Hom(f_U (X), f_U (Y))$ is an epimorphism.
\end{lemma}
\begin{proof}
Consider the canonical sequences for $X$ and $Y$,
$$0 \to t_U (X) \to X \to f_U (X) \to 0$$
and 
$$0 \to t_U (Y) \to Y \to f_U (Y) \to 0$$ 
Applying $\Hom(\ ,f_U (Y))$ to the canonical sequence for $X$ gives
the exact sequence
$$0 \to \Hom(f_U (X), f_U (Y))  \xrightarrow{a} \Hom(X, f_U (Y)) \to \Hom(t_U (X), f_U (Y))$$
Noting that the last term vanishes, this gives that $a$ is an isomorphism.

Applying $\Hom(X,\ )$
 to the canonical sequence for $Y$ gives
the exact sequence
$$\Hom(X, Y)  \xrightarrow{b} \Hom(X, f_U (Y)) \to \Ext^1(X, t_U (Y)).$$
Since $\Hom(U, \tau X)= 0$ we have by Lemma \ref{rigid-rigid} that 
$\Ext^1(X, \Gen U) = 0$, so in particular $\Ext^1(X, t_U (Y)) = 0$. 
Hence the map $b$ is an epimorphism. The induced map
$\alpha = a^{-1} \circ b$ is then also an epimorphism.
\end{proof}

We have the following similar lemma:

\begin{lemma} \label{al1b}
 Let $U,X,Y$ be in $\module\Lambda$ where $U$ is $\tau$-rigid
and $\Hom(U,Y)=0$. Then the induced map $\Hom(X,Y) \rightarrow
\Hom(f_U (X), f_U (Y))$ is an isomorphism.
\end{lemma}
\begin{proof}
Since $\Hom(U,Y) =0$, we have $t_U (Y) = 0$, so $f_U (Y) \simeq Y$.
We have the canonical sequence for $X$:
$$\xymatrix{
0 \ar[r] & t_U(X) \ar[r] & X \ar[r] & f_U(X) \ar[r] & 0 }$$
Applying $\Hom(\ ,Y)$ to this we obtain the exact sequence
$$\xymatrix{
0 \ar[r] &  \Hom(f_U (X), Y) \ar[r] & \Hom(X, Y) \ar[r] & \Hom(t_U (X), Y).}$$
The last term vanishes since $\Hom(U,Y) = 0$ implies that $\Hom(\Gen U,Y) = 0$.
So we have $\Hom(f_U(X), f_U(Y)) \simeq \Hom(f_U (X), Y) \simeq \Hom(X,Y)$.
\end{proof}

Lemma~\ref{al1} has the following consequence in terms of
approximations:

\begin{lemma}\label{al5}
Let $\T$ be a subcategory of $\module \Lambda$. Let $U$ be $\tau$-rigid and assume
$\Hom(U, \tau B) = 0$. If $a \colon B \to A$ is a left $\T$-approximation, then
$f_U(a)\colon f_U(B) \to f_U(A)$ is a left $f_U(\T)$-approximation.
\end{lemma}

\begin{proof}
Let $f_U(T)$ be in $f_U(\T)$, and consider a map $b' \colon f_U (B) \to f_U (T)$.
By Lemma \ref{al1}, there is $b \colon B \to T$ such that $f_U(b) = b'$.
Since $a \colon B \to A$  is a left $\T$-approximation, there is $c \colon A \to T$
such that $b = ca$. It follows that $f_U(b) = f_U(c) f_U(a)$, which proves the claim.
\end{proof}

Lemma~\ref{al5} is used in the proof of part (b) of the following lemma.

\begin{lemma}\label{al2}
Let $U, V$ be in $\module \Lambda$, where $U \amalg V$ is
$\tau$-rigid. Assume no indecomposable summand in $V$ lies
in $\Gen U$ and let $\overline{V} = f_U (V)$.
Let $\T = {^{\perp}(\tau U \amalg \tau V)}$ and let 
$\T' = {^{\perp}(\tau_{J(U)}} \overline{V}) \cap J(U)$.
Then the following hold.
\begin{itemize}
\item[(a)] We have $f_U(\T) = \T'$.
\item[(b)] If $B \to A$ is a left $\T$-approximation in $\module\Lambda$ and $\Hom(U,\tau B)=0$,
then $f_U(B) \to f_U (A)$ is a left $\T'$-approximation (in $\module\Lambda$).
\end{itemize}
\end{lemma}

\begin{proof}
\noindent (a) We first show $f_U(\T) \subseteq \T'$.
Since $U$ is in $\T$, we have $\Gen U \subseteq \T$, and clearly
$\T \subseteq {^{\perp}(\tau U)}$. By \cite[Theorem 3.14]{jasso},
we have that $f_U (\T) = \T \cap U^{\perp}$ is a torsion class
in $J(U)$. So 
$$f_U (\T) = \T \cap  U^{\perp} = {^{\perp}(\tau U \amalg \tau V)} 
\cap  U^{\perp} 
={^{\perp}(\tau V)} \cap J(U) ,$$
and we want to show that 
$f_U (\T)  = {^{\perp}(\tau V)} \cap J(U) \subseteq
 {^{\perp}(\tau_{J(U)} V)} \cap J(U) = \T'$.

Now let $Y$ be in $f_U (\T)$, and consider
the canonical sequence 
$$0 \to t_U (V) \to V \to f_U (V) \to 0$$ 
which, after applying $\Hom(\ , \Gen Y \cap J(U))$ gives rise to an exact sequence
$$\Hom(t_U (V), \Gen Y \cap J(U)) \to 
\Ext^1(\overline{V}, \Gen Y \cap J(U)) \to \Ext^1(V, \Gen Y \cap J(U)).$$
Since $Y$ is in ${^{\perp}(\tau V)}$, we have 
$\Ext^1(V, \Gen Y)= 0$ by Lemma \ref{rigid-rigid}, so in particular
$\Ext^1(V, \Gen Y \cap J(U))= 0$. Since $t_U (V)$ is in $\Gen U$
and $J(U) \subseteq U^{\perp}$, we have that 
$\Hom(t_U (V), \Gen Y \cap J(U)) = 0$. Hence, we also have  
$\Ext^1(\overline{V}, \Gen Y \cap J(U))= 0$, so
$\Ext^1(\overline{V}, \Gen_{J(U)} Y)= 0$ which implies
$\Hom(Y, \tau_{J(U)}\overline{V}) = 0$, by Lemma \ref{rigid-rigid}. Hence we have that 
$Y$ is in
$\T' = {^{\perp}(\tau_{J(U)} V)} \cap J(U)$, which gives
$f_U (\T) \subseteq \T'$.

For full subcategories $\X$ and $\Y$ of $\module \Lambda$, we let
$\X \ast \Y$ denote the full subcategory $$\{M \in \module \Lambda \mid \text{There is an
exact sequence } 0 \to X \to M \to Y \to 0 \text{ with } X \in \X, Y \in \Y \}.$$
Since $f_U (\T) \subseteq \T'$, we have 
$\Gen U \ast f_U (\T) \subseteq \Gen U \ast \T'$.
Since $f_U (\T) = \T \cap U^{\perp}$, it follows from 
\cite[Theorem 3.12]{jasso} that 
$\Gen U \ast f_U (\T) = \T$, so we have 
$\T \subseteq \Gen U \ast \T'$, and we aim to prove equality.

We first claim that $U$ is $\Ext$-projective in 
$\Gen U \ast \T'$. Since $\Hom(U,\tau U)= 0$, we have
$\Ext^1(U, \Gen U) = 0$. We have $\T' \subseteq J(U)$,
so $\Hom(\T', \tau U) = 0$ and hence $\Ext^1(U, \T') = 0$.
From this we obtain that also $\Ext^1(U, \Gen U \ast \T') = 0$,
as required.

We next claim that $V$ is $\Ext$-projective in 
$\Gen U \ast \T'$. Note first that by
\cite[Theorem 3.12]{jasso} we have  
$(\Gen U \ast \T') \cap U^{\perp} = \T'$. 
Since $\overline{V}$ is $\tau_{J(U)}$-rigid in $J(U)$ and
$\T' = {^{\perp}(\tau_{J(U)}} \overline{V}) \cap J(U)$, we have
that $\overline{V}$ is in $\P(\T')$ by~\cite[Theorem 2.10]{air}.
By~\cite[Theorem 3.15]{jasso} 
we have $\P(\T') = f_U\P(\Gen U \ast \T')$, and hence there is
$V'$ in $\P(\Gen U \ast \T')$ such that $\overline{V} = f_U (V')$.
We claim that $V' \amalg U$ is $\tau$-rigid.
Since $V'$ is in $\P(\Gen U \ast \T')$, we have
$\Hom(V', \tau V') = 0$ by \cite[Proposition 1.2(c)]{air}
(noting that $\T'$ is functorially finite in $J(U)$ by~\cite[Theorem 2.10]{air} and
therefore $\Gen U \ast \T'$ is functorially finite in $\module\Lambda$ by~\cite[Theorem 3.14]{jasso}).

Since $\Ext^1(V', \Gen U) = 0$, we have that $\Hom(U,\tau V') = 0$.
We also have that $\Hom(\Gen U, \tau U) = 0$, since $U$ is $\tau$-rigid
and since $\T' \subseteq J(U)$ we have $\Hom(\T', \tau U) = 0$.
Since $V'$ is in $\Gen U \ast \T'$ we hence have $\Hom(V', \tau U)
= 0$, so we have proved the claim that 
 $V' \amalg U$ is $\tau$-rigid.
Since $f_U (\Gen U) = 0$, we may assume that $V'$ has no
direct summands in $\Gen U$. We have $\overline{V} = f_U (V) =
f_U (V')$.
It follows from \cite[Lemmas 5.6, 5.7]{bm} that $\overline{V}$ is basic, since $V$ is basic by assumption. Similarly, also $V'$ is basic and $V \simeq V'$.
So we have proved the claim that $V$ is in $\P(\Gen U \ast \T')$.

Now, using that $U \amalg V$ is in $\P(\Gen U \ast \T')$ in combination with 
\cite[Proposition 2.9]{air}, gives that $\Gen U \ast \T' \subseteq \T={}^{\perp}(\tau U\amalg \tau V)$, and hence we 
obtain  $\T = \Gen U \ast \T'$, which implies $f_U  (\T) = f_U (\Gen U \ast \T') = \T'$,
and this finishes the proof of (a).

\bigskip

\noindent Part (b) follows from part (a) and Lemma~\ref{al5}.
%Let $T'$ be in $\T'$, and let $b' \colon f_U (B) \to T'$
%be a map. By part (a), we have $T' = f_U (T)$ for some $T$ in $\T$.
%Let $a' = f_U(a)$, where $a$ is the left $\T$-approximation $B %\overset{a}{\to} A$.
%
%By Lemma \ref{al1} we have that $\Hom(B,T) \to \Hom(f_U (B), f_U %(T))$ is an epimorphism,
%so there is $B \overset{b}{\to} T$ such that $b' = f_U(b)$. Since %$a$ is by assumption a left $\T$-approximation, there is $A %\overset{c}{\to} T$ such that $b=ca$, and by applying 
%$f_U$ we obtain $f_U (b) = f_U (c) f_U (a)$. This proves the claim %of (b), that $a' =f_U (a)$
%is a $\T'$-approximation.
\end{proof}

\begin{lemma}\label{gen-lem}
Let $U \amalg V$ be $\tau$-rigid in $\module \Lambda$,
let $\T = \Gen (U \amalg V)$ and let $\T' = \Gen f_U (V) \cap J(U)
= \Gen_{J(U)} f_U(V)$. Then$f_U (\T) = \T'$. 
\end{lemma}

\begin{proof}
Since $\Hom(U \amalg V, \tau U) = 0$, we have $\T \subseteq {^{\perp}(\tau U)}$,
so we  have $$\Gen U \subseteq \T \subseteq {^{\perp}(\tau U)}.$$
By \cite[Theorem 3.15]{jasso}, we have that $f_U (\T) = \T \cap U^{\perp}$ is a torsion class in $J(U)$.

Let $Y$ be in $f_U (\T)$ and let $T \in \T$ be such that $Y = f_U (T)$.
There is an epimorphism $U' \amalg V' \xrightarrow{a} T$ with
$U'\in \add U$ and $V'\in \add V$. The canonical maps $U' \amalg V' \xrightarrow{c} 
f_U (V')$
and $T \xrightarrow{d} f_U (T)$ are also epimorphisms, and there is a commutative diagram
$$
\xymatrix{
U' \amalg V' \ar_{c}[d] \ar^{a}[r] & T \ar^{d}[d] \\
f_U (V') \ar^{b}[r] & f_U (T)
}
$$
where $b = f_U(a)$.
\bigskip

\noindent Since $bc = da$ is an epimorphism, also $b$ must be an epimorphism, and hence
$Y = f_U (T)$ is in $\Gen f_U (V)$. Since $f_U (\T) \subseteq J(U)$, we have that $Y$
is in $J(U)$ and hence in $\Gen f_U (V) \cap J(U)$.

Conversely suppose $Y$ is in $\Gen f_U (V) \cap J(U)$. Since $f_U (V)$ is a factor module
of $V$, we have that $Y$ is in $\Gen V$, so $Y$ is in $\T$.
Since $Y$ is in $J(U) \subseteq U^{\perp}$, we hence have that $Y$ is in
$\T \cap U^{\perp} = f_U (\T)$. This finishes the proof of the lemma.
\end{proof}

We also need the following reformulation of Lemma \ref{tau-in-pperp}.

\begin{lemma}\label{al3}
Let $P,V$ be in $\module \Lambda$, with $P$ projective and $\Hom(P,V) = 0$,
and let $\overline{V} = f_P V = V$.
Let $\T = {^{\perp}(\tau V)}$ and let 
$\T' = {^{\perp}(\tau_{P^{\perp}}} \overline{V}) \cap P^{\perp}$.
Then we have $f_P (\T) = \T'$.
\end{lemma}

\begin{proof}
This follows directly from Lemma \ref{tau-in-pperp}, using that $f_{P}(\T) = P^{\perp} \cap
{^{\perp}(\tau V)}$.
\end{proof}

\bigskip

Finally, we need the following. Suppose that $\U$ and $\V$ are objects
in $\C(\Lambda)$, with $\add \U\cap \add \V=0$ and such that $\U\amalg \V$ is support $\tau$-rigid. 
Note that the domain of $\E_{\U\amalg \V}$ is:
$$\{\X\in \ind \C(\Lambda)\,:\, \text{$\X\amalg \U\amalg \V$ support $\tau$-rigid and
$\add \X\cap \add(\U\amalg \V)=0$} \}.$$
Then:

\begin{lemma} \label{l:domain}
Let $\U$ and $\V$ be objects in $\C(\Lambda)$ such that $\U \amalg \V$ is
support $\tau$-rigid and $\add \U\cap \add \V=0$.
Then $\E_{\U}$ induces a bijection between the sets:
$$\{\X\in \ind \C(\Lambda)\,:\, \text{$\X\amalg \U\amalg \V$ support $\tau$-rigid and
$\add \X\cap \add(\U\amalg \V)=0$} \}.$$
and
$$\{\X\in \ind \C(J(\U))\,:\, \text{$\X\amalg \E_{\U}(\V)$ support $\tau$-rigid and
$\add \X\cap \add(\E_{\U}(\V))=0$} \}.$$
\end{lemma}
\begin{proof}
This follows from Theorem~\ref{rigid-sums}.
\end{proof}

\begin{corollary} \label{c:domain}
The composition $\E_{\E_{\U}(\V)}^{J(\U)}\E_{\U}$ is a
well-defined map with domain coinciding with the domain of $\E_{\U \amalg \V}$.
\end{corollary}
\begin{proof}
This follows from Lemma~\ref{l:domain} and the fact that target
set in Lemma~\ref{l:domain} is exactly the domain of $\E_{\E_{\U}(\V)}^{J(\U)}$.
\end{proof}

The following sections will be devoted to proving the following theorem (Thoerem~\ref{main-as} from Section~\ref{main}).

\begin{theorem}\label{main-asso}
Let $\U$ and $\V$ be support $\tau$-rigid objects in $\C(\Lambda)$ with no common direct summands, and suppose that $\U\amalg \V$ is support $\tau$-rigid in $\C(\Lambda)$. Then 
\begin{equation}\label{assoc}
\E_{\E_{\U}(\V)}^{J(\U)}\E_{\U} = \E_{\U \amalg \V}
\end{equation}
\end{theorem}

%\section{Proof of Theorem~\ref{main-asso}}
\noindent \textbf{Proof}
We assume that $\U = U \amalg P[1]$ and $\V= V \amalg Q[1]$, with $U,V,P,Q$ modules and $P,Q$ projective, $\add(\U)\cap \add(\V)=0$ and
$\U\amalg \V$ support $\tau$-rigid.
In view of Corollary~\ref{c:domain}, 
we need to show that
$\E_{\E_{\U}(\V)}^{J(\U)}\E_{\U}(X) = \E_{\U \amalg \V}(X)$
for each indecomposable object $X$ in the domain
%$$\{X \in \ind(\C(\Lambda)) \mid X \amalg \U \text{ } \tau\text{-rigid}\}
%\setminus \ind \U,$$ 
%of $\E_{\U}$.
$$\{X\in \ind \C(\Lambda)\,:\, \text{$X\amalg \U\amalg \V$ support $\tau$-rigid and
$\add X\cap \add(\U\amalg \V)=0$} \}$$
of each of the maps
$\E_{\E_{\U}(\V)}^{J(\U)}\E_{\U}$ and $\E_{\U \amalg \V}$.

Our strategy is to employ a case analysis, based on the properties of $\U$, $\V$ and
$X$, since the maps $\E_{\U}$, $\E_{\E_{\U}(\V)}^{J(\U)}$ and $\E_{\U\amalg \V}$
are defined via cases. We will consider the following cases for $\U$ and $\V$.

\bigskip

\begin{center}
    \begin{tabular}{ | l | l |}
    \hline
    Case I & $\U = U$ and $\V = V$ \\ \hline
    Case II & $\U = U$ and $\V = Q[1]$ \\ \hline
    Case III & $\U = P[1]$ and $\V = V$ \\ \hline
    Case IV & $\U = P[1]$ and $\V = Q[1]$ \\ \hline 
    \end{tabular}
\end{center}

\bigskip

In case II, we assume that $\U = U$ lies in $\module\Lambda$ and that $\V = Q[1]$, where $Q$ lies in $\P(\Lambda) \cap {^{\perp}U}$.
In this case the claim that equation~\eqref{assoc} holds follows directly, since we have 
$$\E_{\U \amalg \V} = \E_{U \amalg P[1]} = \E^{J(U)}_{\E_U(P[1])} \E_U =\E^{J(U)}_{\E_U(\V)} \E_U = \E^{J(\U)}_{\E_{\U}(\V)} \E_{\U}$$
where the second equality holds by definition of 
$\E_{U \amalg P[1]}$. So it remains to consider cases I, III and
IV.

For each of the cases I, III and IV we will also need to further subdivide according to the properties of $X$ in $\ind \C(\Lambda)$.
We consider Case I in Section~\ref{s:caseI}, Case III in Section~\ref{s:caseIII}
and Case IV in Section~\ref{s:caseIV}.
Finally, we must consider the `mixed case', where $\U$ and
$\V$ have both module and shifted projective direct summands;
this is considered in Section~\ref{s:mixed}.

\section{Proof of Theorem~\ref{main-asso}, Case I}
\label{s:caseI}

We assume that $\U = U$ and $\V = V$ for $U, V$ in $\module \Lambda$ where $U \amalg V$ is $\tau$-rigid.
We divide Case I into the following subcases.
\begin{itemize}
\item Case I*: $\U = U$ and $\V = V$ where $U$ and $V$ lie in $\module \Lambda$ and $\add V \cap \Gen U = 0$. 
\item Case I**: $\U = U$ and $\V = V$ where $U$ and $V$ lie in $\module \Lambda$ and $\V \in \Gen U$.
\end{itemize}

We firstly note that Case I will follow from these two cases:

\begin{proposition}
Assume that (\ref{assoc}) holds in both cases I* and I**. Then (\ref{assoc}) holds in Case I.
\end{proposition}

\begin{proof}
Write $V = V_1 \amalg V_2$, where $V_1$ is in $\Gen U$, and $\add V_2 \cap 
\Gen U = 0$.
Then $\E_{U \amalg V} = \E_{(U \amalg V_1) \amalg V_2}$.
If $U = 0$, then also $V_1 = 0$, and the result is trivial.
We therefore assume $U \neq 0$. We proceed by induction on $n = r(\module \Lambda)$. Hence we can assume that it holds for $J(U)$, since $r(J(U)) < n$.

Note that we have $\Gen (U \amalg V_1) = \Gen U$,
so $\add V_2 \cap \Gen(U \amalg V_1) = 0$.
Hence we have:
\begin{align}
\nonumber \E_{U \amalg V} &= \E_{(U\amalg V_1)\amalg V_2} \\ \label{l1}
&=\E^{J(U \amalg V_1)}_{\E_{U \amalg V_1}(V_2)} \E_{U \amalg V_1} \\[5pt]
\label{l2}  &= \E^{J(U \amalg V_1)}_{\E_{U \amalg V_1}(V_2)} \E_{\E_U(V_1)}^{J(U)} \E_U \\[5pt]
\label{l3} &= \E^{J_{J(U)}(\E_U(V_1))}_{\E_{U \amalg V_1}(V_2)} \E_{\E_U(V_1)}^{J(U)} \E_U \\[5pt]
\label{l4} &= \E^{J_{J(U)}(\E_U(V_1))}_{\E^{J(U)}_{\E_U(V_1)}(\E_U(V_2))} \E_{\E_U(V_1)}^{J(U)} \E_U \\[5pt]
\label{l5} &= \E^{J(U)}_{\E_U(V_1) \amalg \E_U(V_2)} \E_U \\[5pt]
\label{l6} &=  \E^{J(U)}_{\E_U(V_1 \amalg V_2)} \E_U \\[5pt]
\nonumber &=  \E^{J(U)}_{\E_U(V)} \E_U
\end{align}
where~\eqref{l1} holds by Case I*, the equations~\eqref{l2} and~\eqref{l4} hold
by Case I**, and \eqref{l3} holds by Theorem \ref{main-comp}.
Furthermore (\ref{l5}) holds (in $J(U)$) by the induction assumption, while (\ref{l6}) holds by definition.
\end{proof}

For each of the subcases I* and I**, 
we will need to consider the following cases for $X$.

\begin{itemize}
\item[(a)]  $X \in \ind \Lambda$ and $X \not \in \Gen(U \amalg V)$
\item[(b)]  $X \in \ind \Lambda$ and $X \in \Gen(U \amalg V) \setminus \Gen U$
\item[(c)]  $X \in \ind \Lambda$ and $X \in \Gen U$
\item[(d)]  $X \in \ind \P(\Lambda)[1]$
\end{itemize}

\bigskip

\noindent{\bf Case I*:}
We assume that $\U = U$ and $\V = V$ where $U$ and $V$ lie in $\module \Lambda$ and $\add V \cap \Gen U = 0$, i.e. $V$ has no
direct summands in $\Gen U$. We set $\overline{V} = f_U (V)$.

\bigskip

\begin{lemma}\label{a-case}
With the above assumptions on $U$ and $V$, we have that $f_U(X)$ is not
in $\Gen \overline{V}$ and that
$$f^{J(U)}_{\overline{V}} (f_U (X))\simeq f_{U \amalg V}(X)$$
for any $X$ not in $\Gen(U \amalg V)$.
\end{lemma}

\begin{proof}
Consider the composition
$$X \xrightarrow{a} f_U (X) \xrightarrow{b} f^{J(U)}_{\overline{V}}f_U (X)$$
We first claim that $ba$ is a left $(U\amalg V)^{\perp}$-approximation.
Let $c \colon X \to Y$ be a map, with $Y$ in $(U\amalg V)^{\perp}$.
Since $Y$ is in $U^{\perp}$, and $a$ is left $U^{\perp}$-approximation, there is a map
$d \colon f_U (X) \to Y$ such that $da = c$. Applying $\Hom(\ , Y)$ to the canonical
sequence 
$$0 \to t^{J(U)}_{\overline{V}}f_U (X) \xrightarrow{e} f_U (X) \xrightarrow{b} f^{J(U)}_{\overline{V}}f_U (X) \to 0$$
of $f_U (X)$ gives the exact sequence
\begin{equation}\label{seq2}
\Hom(f^{J(U)}_{\overline{V}}f_U (X), Y) \to \Hom(f_U (X),Y) \to \Hom(t^{J(U)}_{\overline{V}}f_U (X),Y).
\end{equation}
We have that $t^{J(U)}_{\overline{V}}f_U (X)$ is in $\Gen_{J(U)} \overline{V}
= \Gen \overline{V} \cap J(U) \subseteq \Gen V$, since $\overline{V} = f_U (V)$
is a factor module of $V$. Since $Y$ is in $V^{\perp}$, we then have that $de = 0$.
Hence, by the sequence (\ref{seq2}), there is a map $g \colon f^{J(U)}_{\overline{V}}f_U (X)
\to Y$, such that $g b = d$. Hence $c = g ba$, which proves that $ba$ is 
a left $(U\amalg V)^{\perp}$-approximation as claimed. We have that $ba$ is minimal,
since both $a$ and $b$ are epimorphisms.

The canonical map $X \to f_{U \amalg V}(X)$ is a also a minimal left $(U\amalg V)^{\perp}$-approximation. So we have  
\begin{equation}\label{iso1}
 f^{J(U)}_{\overline{V}}f_U (X) \simeq f_{U \amalg V}(X)
\end{equation}

In particular, since by assumption $X$ is not in $\Gen (U \amalg V)$, we have that 
$f^{J(U)}_{\overline{V}}f_U (X) \simeq f_{U \amalg V}(X)$ in non-zero, so
$f_U (X)$ is not in $\Gen \overline{V}$. 
\end{proof}

\bigskip

\noindent {\bf Case I* (a):}  
We assume that $X$ is indecomposable in $\module \Lambda$,
that $X \amalg U  \amalg V$ is a $\tau$-rigid module, and that
$X$ does not lie in $\Gen(U \amalg V)$.

\bigskip

We then have that $\E_U(X) = f_U(X)$, since $X$ does
not lie in $\Gen U$.
Using the first claim of Lemma \ref{a-case} we have that $f_U(X)$ is not in $\Gen \overline{V}$.
Hence,
$$\E^{J(U)}_{\overline{V}} (\E_U X) =  f^{J(U)}_{\overline{V}} (f_U (X)).$$
Since $X$ is not in $\Gen (U \amalg V)$, we have $\E_{U \amalg V}(X) = f_{U \amalg V}(X)$, and
equation~\eqref{assoc} now follows from the second claim of Lemma \ref{a-case}.
 
\bigskip

\noindent {\bf Case I* (b):}  
We assume that $X$ is indecomposable in $\module \Lambda$, that 
$X \amalg U  \amalg V$ is a $\tau$-rigid module, that
$X$ is in $\Gen(U \amalg V)$, and that $X$ does not lie in $\Gen(U)$.

\bigskip

We have $\E_U(X) = f_U(X)$, since by assumption, $X$ is not in $\Gen U$.
By Lemma \ref{gen-lem}, we have that $f_U \Gen(U \amalg V) = \Gen f_U(V) \cap J(U)=
\Gen_{J(U)}(f_U (V))$.
Hence we have that $f_U (X)$ is in $\Gen_{J(U)}(f_U (V))$.

There is a right exact sequence 
\begin{equation}\label{seq-4}
Y_{f_U(X)} \to \overline{V}'_{f_U(X)} \to f_U(X) \to 0,
\end{equation}
where the first map is a minimal left $\add U$-approximation, the second map is a
minimal right $\add U$-approximation, and
$Y_{f_U(X)}$ lies in $\P({^{\perp}(\tau_{J(U)} \overline{V})} \cap J(U))$.
We then have $\E_{\overline{V}}^{J(U)}\E_U(X) = f_{\overline{V}}^{J(U)} (Y_{f_U(X)})$.

To compute $\E_{U \amalg V}(X)$, consider the right exact sequence
\begin{equation}\label{seq-3}
Y_{X}' \to {U}' \amalg  {V}' \xrightarrow{\alpha} X \to 0
\end{equation}
where the first map is a minimal left $\add (U \amalg V)$-approximation, the second map is a
minimal right $\add (U \amalg V)$-approximation and
$Y_X'$ lies in $\P({^{\perp}(\tau U \amalg \tau V}))$.
We then have $\E_{U \amalg V}(X) = f_{U \amalg V} (Y_X')$.

We now aim to prove the following.
\begin{claim}\label{fu-claim}
Applying $f_U$ to the right exact sequence (\ref{seq-3}) gives the right exact sequence (\ref{seq-4}).
\end{claim}

To prepare for the proof of Claim \ref{fu-claim}, we consider first a 
more general set-up.

\begin{lemma}\label{prep}
Let $(\T, \F)$ be an arbitrary torsion pair in $\module \Lambda$. Assume that there is commutative diagram
$$
\xymatrix{
A \ar[r]^{a} \ar[d]^x & B \ar[r]^{b} \ar[d]^y & C \ar[d]^{z} \ar[r] & 0 \\
A' \ar[r]^{a'} & B' \ar[r]^{b'} & C' \ar[r] & 0
}
$$
where the vertical maps are minimal left $\F$-approximations (and hence epimorphisms),
and the upper row is a right exact sequence. Then the map $b'$ is an epimorphism and
 for any $Z$
in $\F$, and any map $t \colon B'\to Z$ with $t a' = 0$, there is a map $u \colon C' \to Z$,
such that $u b' = t$.
\end{lemma}

\begin{proof}
For the first claim, note that $zb =  b' y$ is an epimorphism, hence also $b'$ is an
epimorphism. 
We have that $zba = 0$, and this implies $b'a'x = 0$ and hence $b'a' = 0$,
since $x$ is an epimorphism. Now $ta' = 0$ implies $ta' x = 0$, and hence $tya = 0$. Since $b$ is the cokernel of $a$, there is a map $u' \colon C \to Z$
such that $ty = u'b$. Since $z$ is an $\F$-approximation, and by assumption $Z$ is 
in $\F$, there is $u \colon C' \to Z$ such that $u' = uz$. It then follows that
$ty = u' b = uzb = u b' y$, and since $y$ is an epimoprhism, we have 
$t = u b'$, as claimed. 
\end{proof}

\begin{proof}[Proof of Claim \ref{fu-claim}]
Apply $f_U$ to the sequence~\eqref{seq-3}, and consider the commutative diagram
$$
\xymatrix{
& 0 \ar[d] & 0 \ar[d]  & & \\
& t_U(U' \amalg V') \ar[r]^{\gamma} \ar[d]^c & t_U(X) \ar[d]^d & \\ 
Y_X' \ar[r]^{s'} \ar[d]^r & U' \amalg V' \ar[r]^{\alpha} \ar[d]^p & X \ar[d]^{q} \ar[r] & 0 \\
f_U(Y_X') \ar[r]^{s}  \ar[d] & f_U(V') \ar[r]^{\beta}  \ar[d]  & f_U(X) \ar[r]  \ar[d] & 0 \\
0 & 0 & 0 &
}
$$
where the second row is sequence (\ref{seq-3}), so is exact, and the second and third columns are
the canonical sequences for $U' \amalg V'$ and $X$, respectively. Note that the
map $\gamma$ exists since $q \alpha c = \beta p  c = 0$, so $\alpha c$ factors through
$d$.

We first claim that all
objects in the third row are in $J(U)$. This follows from the fact that 
all objects in sequence (\ref{seq-3}) are in ${^{\perp}(\tau U)}$, and hence the same
hold for all objects in the third row, since ${^{\perp}(\tau U)}$ is closed under
factor objects. All objects in the third row are by definition in $U^{\perp}$, and
hence also in $J(U) =  U^{\perp} \cap {^{\perp}(\tau U)}$.

We next claim that $\beta$ is the cokernel of $s$. We have $\beta s = f_U(\alpha s') = f_U(0) = 0$.
It now follows from applying Lemma \ref{prep}, with the
torsion pair $(\Gen U, U^{\perp})$, and using that $J(U) \subseteq U^{\perp}$ 
that $\beta$ is the cokernel of $s$ in $J(U)$, that is the sequence
\begin{equation}\label{seq-5}
f_U(Y_X') \xrightarrow{s} f_U(V')  \xrightarrow{\beta} f_U(X) \to 0
\end{equation}
is exact in $J(U)$ (and hence also an exact sequence in $\module \Lambda$).

We now claim that the map $s$ is a minimal left $\add \overline{V}$-approximation.
For this let $b \colon f_U (Y_X') \to \overline{V}''$ be a map with 
$ \overline{V}''$ in $\add \overline{V}$. Let $V''$ in  $\add V$ be such that 
$f_U(V'') =  \overline{V}''$, and let $g \colon V'' \to  \overline{V}''$ be the
canonical epimorphism.
Consider the canonical exact sequence for $V''$,
$$0 \to t_U (V'') \to V''  \xrightarrow{g} \overline{V}'' \to 0$$
and note that since $\Gen U \subseteq {^{\perp}(\tau U \amalg \tau V)}$, all terms
in the sequence are in ${^{\perp}(\tau U \amalg \tau V)}$.
Applying $\Hom(Y_X', \ )$ 
gives the exact sequence
$$\Hom(Y_X',V'') \to \Hom(Y_X', \overline{V}'') \to \Ext^1(Y_X', t_U (V'')).$$
Using that $Y_X'$ is in $\P({^{\perp}(\tau U \amalg \tau V)})$ and that 
$t_U(V'') \in \Gen U \subseteq {}^{\perp}(\tau U \amalg \tau V)$, we have that 
the last term vanishes, and hence the first map is surjective, and therefore there
 is a map $f' \colon Y_X' \to V''$
such that $g f' = br$.

The map $s'$ is an $\add(U \amalg V)$-approximation, therefore there is a map 
$f \colon U' \amalg V' \to V''$, such that 
$f' = fs'$. Then we have $br = gf' = gf s'$. 
Now consider the canonical sequence 
\begin{equation}
\label{e:canU}
0 \to t_U(U' \amalg V') \to U' \amalg V' \xrightarrow{p} f_U(U' \amalg V') = f_U(V') 
\to 0
\end{equation}

Note that
$\Hom( t_U(U' \amalg V'), \overline{V}'') =0$ since
$\overline{V}$ is in ${^\perp(\tau U)}$.
Applying $\Hom(\ , \overline{V}'')$ to~\eqref{e:canU}, we obtain
$$\Hom(f_U(V'), \overline{V}'') \simeq \Hom(U' \amalg V',  \overline{V}'').$$
Hence there is a map $e \colon f_U(V') \to \overline{V}''$, such that $ep = gf$.
We then have $br= gfs' = eps' = esr$, that is $(b-es)r= 0$. Since $r$ is an epimorphism,
this implies $b= es$, and we have proved the claim that $s$ is a left
$\add \overline{V}$-approximation.

We next claim that $s$ is a left minimal map. Let $$s= \begin{pmatrix}
s_1 \\0
\end{pmatrix} \colon f_U (Y_X') \to f_U (V')$$ where 
$s_1$ is left minimal. Then $\coker(s) \simeq  \coker(s_1) \amalg M$
for some module $M$. Note that
$X$ indecomposable implies that $\coker (s) \simeq f_U(X)$ is indecomposable by \cite[Lemma 4.6]{bm},
hence we have $\coker(s_1) = 0$ or $M=0$. 
If $\coker(s_1) = 0$, then $f_U (X)$ is in $\add f_U(V')$. 
Then there is an indecomposable direct summand $V_0$ of $V$, such that
$f_U (X) \simeq f_U(V_0)$. But then, by \cite[Lemma 5.7]{bm} we have 
$X \simeq V_0$, but this is a contradiction, since $X$ is by assumption not in $\add V$. 
Hence $M= 0$, and therefore $s= s_1$ is left minimal.

We claim that $f_U(Y_X')$ is in $\P(\T')$, where $\T' = {^\perp(\tau_{J(U)} 
\overline{V}) \cap J(U)}$ is a torsion class in $J(U)$. For this 
consider the torsion class $\T = {^\perp(\tau U \amalg \tau V)}$ in $\module \Lambda$.
By Lemma \ref{al2} we have that $f_U(\T) = \T'$.
It then follows from \cite[Theorem 3.15]{jasso}, that $f_U(\P(\T)) = \P(\T')$, and hence
$f_U(Y_X')$ is in $\P(\T')$, since $Y'_X\in \P(T)$.

We can now apply \cite[Proposition 4.7]{bm} to obtain that the sequences (\ref{seq-4}) and (\ref{seq-5}) are isomorphic, and this concludes the proof of the claim.
\end{proof}

Using Claim \ref{fu-claim} we obtain
$$\E_{\overline{V}}^{J(U)} \E_U (X) = \E_{\overline{V}}^{J(U)} f_U (X) = 
f_{\overline{V}}^{J(U)}(Y_{f_U(X)})=f_{\overline{V}}^{J(U)}f_U (Y_X').$$

Moreover, we have that $Y_X'$ is not in $\Gen(U \amalg V)$, since 
$\E_{U \amalg V}(X) = f_{U \amalg V}(Y_X') \neq 0$. Therefore, using 
Lemma \ref{a-case}, we obtain
$$\E_{U \amalg V}(X) = f_{U \amalg V} (Y_X') \simeq f_{\overline{V}}^{J(U)} f_U(Y'_X).$$
This finishes the proof that equation (\ref{assoc}) holds in this case.

\bigskip

\noindent {\bf Case I* (c):} 
We assume that $X$ is an indecomposable module in
$\module \Lambda$, that $X \amalg U  \amalg V$ is a $\tau$-rigid
module and that $X$ is in $\Gen(U)$.

\bigskip

Let $\T= {^{\perp}(\tau U \amalg \tau V)}$ and $\T' = {^{\perp}(\tau_{J(U)} \overline{V})}
\cap J(U)$, and consider the exact sequence 

\begin{equation}\label{esq-1}
Y_X' \to U_X \to X  \to 0,
\end{equation}
where the first map is a minimal left $\add U$-approximation, and the second map 
is a minimal right $\add U$-approximation.
We then have that $\E_{\U}(X) = \E_U(X)= f_U(Y_X')[1]$. Note that by Theorem
 \ref{rigid-sums} the object
$\E_{U}(X \amalg V) =f_U(Y_X')[1] \amalg \overline{V}$  is support $\tau$-rigid in $\C(J(\U))$, and  hence
we have that $f_U(Y_X')$ is in ${^{\perp} \overline{V}}$.

We have $\E^{J(\U)}_{\E_{\U}(\V)}\E_{\U}(X) = \E^{J(\U)}_{\overline{\V}}\E_{U}(X) =
f^{J(U)}_{\overline{V}}Y_X''$, where $f_U (Y_X') \to Y_X''$ is a minimal left
$\T'$-approximation.

To compute $\E_{\U \amalg \V}(X)$ we consider the right exact sequence
\begin{equation}\label{esq-2}
Y_X''' \to U_X' \amalg V_X' \to X  \to 0
\end{equation}
where the first map is a minimal left $\add (U \amalg V)$-approximation, and the second map 
is a minimal right $\add (U \amalg V)$-approximation.
Then $\E_{\U \amalg \V}(X) = \E_{U \amalg V}(X) = f_{U \amalg V}(Y_X''')$.

Since $\E_{U \amalg V}(X) \neq 0$, we have that $Y_X'''$ is not in $\Gen(U \amalg V)$. 
By Lemma \ref{a-case}, we hence have that 
$\E_{U \amalg V}(X) = f_{U \amalg V}(Y_X''')[1] = f_{\overline{V}}^{J(U)} f_U(Y_X''')[1]$.

It is therefore sufficient to prove that 
$$f_{\overline{V}}^{J(U)} (Y_X'') \simeq f_{\overline{V}}^{J(U)} f_U(Y_X''')$$
The main steps are as follows:

\begin{claim} 
\begin{itemize}
\item[(a)] There is a map $Y_X' \to U_X \amalg Y_X'''$, which is a left $\T$-approximation. 

\item[(b)] The induced map $f_U(Y_X') \to f_U(Y_X''')$ is a left $\T'$-approximation.

\item[(c)] We have that $Y_X''$  is a direct summand in $f_U(Y_X''')$
\item[(d)] We have $f_{\overline{V}}^{J(U)} (Y_X'') \simeq f_{\overline{V}}^{J(U)} f_U(Y_X''')$.
\end{itemize}
\end{claim}

\begin{proof}
\noindent (a):
Consider the diagram
$$
\xymatrix{
\PP_{Y_X'} \ar[r] &  \PP_{U_X}  \ar[r] &  \PP_{X}   \ar[r]  \ar^{=}[d] & \PP_{Y_X'}[1] \\
\PP_{Y_X'''}  \ar[r] & \PP_{U_X' \amalg V_X'}  \ar[r] &  \PP_{X}  \ar[r] & \PP_{Y_X'''}[1]
}$$
where the rows are triangles giving rise (by taking homology) to
the exact sequences (\ref{esq-1}) and (\ref{esq-2}), respectively (see Section
\ref{bi}).
We have that $U, Y_X'''$ are in $\P(\T)$, so in particular $\Hom(Y_X''', \tau U) = 0$.
Hence by Lemma \ref{rigid-rigid}
we have that $\Hom(\PP_{U_X}, \PP_{Y_X'''}[1]) = 0$.
Therefore (see \cite[Section 1.4]{nee}) the above diagram can be completed to a commutative diagram in 
such a way that there is an induced triangle
$$\PP_{U_X' \amalg V_X'}[-1] \xrightarrow{g} \PP_{Y_X'} \xrightarrow{h} \PP_{U_X' \amalg Y_X'''} \to  \PP_{U_X' \amalg V_X'}.$$
Now, let $k \colon \PP_{Y_X'} \to \PP_T$ be a map with $T \in \T= {^{\perp}(\tau U \amalg \tau V)}$. Then $\Hom(T, \tau U \amalg \tau V) = 0$, and hence by Lemma \ref{rigid-rigid},
we have $\Hom(\PP_{U_X' \amalg V_X'}, \PP_T[1])= 0$. Hence we have $kg = 0$, so 
by exactness of $\Hom(\ ,\PP_T)$ we have that there is map $l \colon 
\PP_{U_X' \amalg Y_X'''} \to \PP_T$, such that $lh  =k$.
It then follows that the map
$H^0(k) \colon Y_X \to T$ factors through $H^0(l)$,
and by Lemma \ref{rigid-rigid}, we have that any map $Y_X' \to T$ factors through
$Y_X' \to U_X' \amalg Y_X'''$. This finishes the proof of the claim.

\bigskip

\noindent (b): This follows directly from Lemma \ref{al2}, using that $U, Y_X''' \in \P(\T)$
and therefore $\Hom(U, \tau Y_X''')= 0$. 

\bigskip

\noindent (c): This follows directly from (b), using that $f_U(Y_X') \to Y_X''$ is a 
minimal left $\T'$-approximation.

\bigskip

\noindent (d): This follows directly from (c), using that both  $f_{\overline{V}}^{J(U)} Y_X''$ and $f_{\overline{V}}^{J(U)} f_U(Y_X''')$ are indecomposable.
\end{proof}

So equation (\ref{assoc}) is proved for this case.

\bigskip

\noindent {\bf Case I* (d):} 
We assume that $X$ is of the form $R[1]$, where $R$ is an
indecomposable module in $(\P(\Lambda) \cap {^{\perp}(U \amalg V)})$.

\bigskip

Let $\T = {^\perp(\tau U \amalg \tau V)}$ and let 
$\T' = {^\perp(\tau_{J(U)} 
\overline{V}) \cap J(U)}$.

We first compute $\E_{U \amalg V}(X)$. For this, let $R \to Y_R$ be a 
minimal left $\T$-approximation. Then $\E_{U \amalg V}(X) = f_{U \amalg V} (Y_R)[1]
= f^{J(U)}_{\overline{V}} f_U(Y_R)[1]$, where the last equation follows from
Lemma \ref{a-case}.

Similarly, we compute $\E_{U}(X)$ by letting $R \to Y_R'$ be a 
minimal left $\T$-approximation and then $\E_U(X) = f_U (Y_R')[1]$.

To compute $\E_{\overline{V}}^{J(U)} \E_U(X)$, let $f_U (Y_R') \to Y_R''$ be a
minimal $\T'$-approximation. Then we have $\E_{\overline{V}}^{J(U)} \E_U(X) = 
f_{\overline{V}}^{J(U)} (Y_R'')[1]$.

\begin{claim}\label{c:iso}
We have that $f_U (Y_R) \simeq Y_R''$.
\end{claim}

\begin{proof}
We first claim that the composition $R \xrightarrow{h} Y_R  \xrightarrow{k} f_U(Y_R)$
is a left $\P(\T')$-approximation.
Consider a map $f \colon R \to N'$ with $N'$ in $\P(\T')$.
By Lemma \ref{al2}(b), we have that $f_U \P(\T) = \P(\T')$, so there is a module
$N$ in $\P(\T)$ satisfying $f_U(N) = N'$.
Since $R$ is projective, there is a map $u \colon R \to N$, such 
that $gu = f$. 
Since $h \colon R \to Y_R$ is a $\P(\T)$-approximation and $N$ is in $\P(\T)$,
there is $v \colon Y_R \to N$ such that $u =vh$.
Note that $N$ is in $\P(\T) \subseteq J(U) \subseteq U^{\perp}$.
Since $k$ is a left $U^{\perp}$-approximation, there is $w \colon f_U (Y_R) \to N$
such that $v = wk $. So, we have $f = gu = gvh = gwhk$.
Note that $N'$ is in $\P(\T') \subseteq J(U)\subseteq U^{\perp}$.
Since $k$ is a left $U^{\perp}$-approximation, there is $w \colon f_U (Y_R) \to N$
such that $gv = wk $. So, we have $f = gu = gvh = wkh$,
and hence $kh$ is a  left $\P(\T')$-approximation.

Next, we claim that the composition $$R \xrightarrow{b} Y_R'  \xrightarrow{c} f_U(Y_R')
\xrightarrow{d} Y_R''$$
is a left $\P(\T')$-approximation.
Let $N$ be in  $\P(\T')$, and let $a \colon R \to N$ be a map.
Since $N$ is in ${^{\perp}(\tau U)}$ and $b$ is a left ${^{\perp}(\tau U)}$-approximation, there is a map $a' \colon Y_R' \to N$ such that $a = a' b$.   
Since $N$ is in ${U^{\perp}}$ and $c$ is a left  ${U^{\perp}}$-approximation, there is a map $a'' \colon f_U(Y_R') \to N$ such that $a' = a'' c$.   
Since $N$ is in $\T'$ and $d$ is a left  $\T'$-approximation, there is a map $a''' \colon Y_R'' \to N$ such that $a'' = a''' d$.   
So we have $a'''dcb = a''cb = a'b = a$, so $dcb$ is a left $\P(\T')$-approximation as claimed.

Note that both $Y_R$ and $Y_R''$ are indecomposable by \cite[Prop. 3.7]{bm}, hence
also $f_U(Y_R)$ is indecomposable, by \cite[Lemma 4.6]{bm}.
It then follows that both
 $$X \xrightarrow{h} Y_R  \xrightarrow{k} f_U(Y_R)$$ and
 $$X \xrightarrow{b} Y_R'  \xrightarrow{c} f_U(Y_R')
 \xrightarrow{d} Y_R''$$
 are minimal left $\T'$-approximations.
 So we obtain $f_U(Y_R) \simeq Y_R''$.
  \end{proof}
By Claim~\ref{c:iso} we now have that   
$$\E_{U \amalg V}(X) = f_{\overline{V}}^{J(U)}f_U(Y_R)[1] = f_{\overline{V}}^{J(U)}(Y_R'')
 [1] =
 \E_{\overline{V}}^{J(U)} \E_U(X)$$
 and hence equation (\ref{assoc}) holds also in this case.
  
\bigskip 
 
\noindent {\bf Case I**:}
We assume that $\U = U$ and $\V = V$ for $U, V$ in
$\module \Lambda$, that $U \amalg V$ is a $\tau$-rigid module,
and that $V$ lies in $\Gen U$.

\bigskip 

Let $\overline{V} = \overline{Q}[1]$, where $\overline{Q}$ is in $\P(J(U))$. We first make the following observation.

\begin{lemma}\label{tf}
With the assumptions of Case I**, we have that $f_{U \amalg V} = f_U$.
\end{lemma}

\begin{proof}
Since $V \in \Gen U$, we have that $\Gen(U \amalg V) = \Gen U$. 
It follows that $t_{U \amalg V} = t_U$, and then by uniqueness of canonical 
sequences, that also $f_{U \amalg V} = f_U$.
\end{proof}

\bigskip
 
\noindent {\bf Case I** (a):}
We assume that $X$ is an indecomposable module in $\module \Lambda$, that $X \amalg U \amalg V$ is a $\tau$-rigid module,
and that $X$ does not lie in $\Gen (U \amalg V) = \Gen U$.

\bigskip

We have that $\E_U(X) = f_U(X)$, and 
$$\E_{\overline{V}}^{J(U)}\E_U(X) = \E_{\overline{V}}^{J(U)} (f_U(X)) = 
\E_{Q[1]}^{J(U)}(f_U(X)) =f_U(X).$$
 Note that the last equation holds since $\overline{Q}$ is in $\P(J(U))$ and
 we have $\Hom(\overline{Q}, f_U(X)) = 0$ since $\E_U(V \amalg X)$ is
support $\tau$-rigid in $\C(J(U))$ by Theorem \ref{rigid-sums}.

By Lemma \ref{tf} we have
$$\E_{U \amalg V}(X) = f_{U \amalg V}(X) =f_U(X) = \E_{\overline{V}}^{J(U)}\E_U(X) $$
 and the claim that equation (\ref{assoc}) holds, is proved in this case. 

\bigskip

\noindent {\bf Case I** (b):}
Since $\Gen(U \amalg V) = \Gen U$, this case (where $X$ lies in
$\ind \Lambda$ and $X\in \Gen(U\amalg V)\setminus \Gen U$)
cannot occur.

\bigskip

\noindent {\bf Case I** (c):} 
We assume that $X$ is an indecomposable $\tau$-rigid module,
that $X \amalg U \amalg V$ is a $\tau$-rigid module, and that
$X$ lies in $\Gen U = \Gen(U \amalg V)$.

\bigskip
 
In order to compute  $\E_{U}(X)$, we consider the exact sequence
\begin{equation}\label{es-1}
Y_X \to U_X \to X \to 0
\end{equation}
where the first map is a minimal left $\add U$-approximation, and
the second map is a minimal right $\add U$-approximation.
Then $\E_{U}(X) = f_U(Y_X)[1]$.

We have 
$$\E^{J(U)}_{\E_{\U}(\V)} \E_{\U}(X) = \E^{J(U)}_{\E_{\overline{\V}}} \E_{\U}(X) =
 f^{J(U)}_{\overline{Q}} f_U(Y_X)[1].$$

Next, to compute  $\E_{U \amalg V}(X)$, we consider the exact sequence
\begin{equation}\label{es-2}
Y_X' \to U_X' \amalg V_X' \to X \to 0,
\end{equation}
where the first map is a minimal left $\add (U \amalg V)$-approximation, and
the second map is a minimal right $\add (U \amalg V)$-approximation.
Then  $\E_{U \amalg V}(X) = f_{U \amalg V}(Y_X')[1]$.  

By Lemma~\ref{tf}, it now follows that 
$$\E_{U \amalg V}(X) = f_{U \amalg V}(Y_X')[1] =  f_{U}(Y_X')[1].$$ 

By the above, it will be sufficient to prove that 
\begin{equation}\label{eqv-1}
  f^{J(U)}_{\overline{Q}} f_U(Y_X) \simeq f_{U}(Y_X').
 \end{equation}

The main steps in the proof are as follows:

\begin{claim}\label{c:2c} 
Let $\T = {^{\perp}(\tau U \amalg \tau V)}$ and let 
$\T' =  {\overline{Q}}^{\perp} \cap J(U)$. Then the following hold.
\begin{itemize}
\item[(a)] We have $f^{J(U)}_{\overline{Q}} f_U(\T) =  f_U(\T) = \T'$.
\item[(b)] There is a map $Y_X \xrightarrow{\gamma} Y_X' \amalg U_X$ which is a left $\T$-approximation.
\item[(c)] The map $f_U(Y_X) \xrightarrow{f_U(\gamma)} f_U(Y_X')$ is a left $\T'$-approximation.
\item[(d)] The map $f^{J(U)}_{\overline{Q}}f_U(Y_X) \xrightarrow{f^{J(U)}_{\overline{Q}}f_U(\gamma)}
 f^{J(U)}_{\overline{Q}}f_U(Y_X')$ is a left $f^{J(U)}_{\overline{Q}} \T' = \T'$ -approximation.
\item[(e)] The map $f^{J(U)}_{\overline{Q}}f_U(\gamma)$ is an isomorphism.
 \item[(f)] We have $f^{J(U)}_{\overline{Q}}f_U(Y_X') \simeq f_U(Y_X')$.
 \item[(g)] We have $f^{J(U)}_{\overline{Q}} f_U(Y_X) \simeq f_{U}(Y_X')$.
  \end{itemize}
\end{claim}

\begin{proof}
\noindent (a):
Since $V$ is in $\Gen U$, we have by Lemma \ref{tf}
that $f_U = f_{U \amalg V}$, and we have
\begin{multline*}
 f_U(\T) = f_{U \amalg V}(\T) = f_{U \amalg V}({^{\perp}(\tau U \amalg \tau V)}) =
(U \amalg V)^{\perp} \cap {^{\perp}(\tau U \amalg \tau V)} \\
=J(U \amalg V) \overset{(\ast)}{=} J_{J(U)}(\E_U(V)) = J_{J(U)}(\overline{Q}[1]) =\overline{Q}^{\perp} 
\cap J(U) = \T'
\end{multline*}
where $(\ast)$ holds by Theorem \ref{main-comp}. 
This proves the second equality. But $f^{J(U)}_{\overline{Q}}$ clearly acts
as the identity on objects in $\overline{Q}^{\perp} 
\cap J(U)$, and this proves the first equality.

\bigskip

\noindent (b):
Consider the diagram
$$
\xymatrix{
\PP_{Y_X} \ar[r] &  \PP_{U_X}  \ar[r] &  \PP_{X}   \ar[r]  \ar^{=}[d] & \PP_{Y_X}[1] \\
\PP_{Y_X'}  \ar[r] & \PP_{U_X' \amalg V_X'}  \ar[r] &  \PP_{X}  \ar[r] & \PP_{Y_X'}[1]
}$$
where the rows are triangles giving rise (by taking homology) to
the exact sequences (\ref{es-1}) and (\ref{es-2}), respectively (see Section 2 for details).
Since $Y_X'$ is in $\T \subseteq {^{\perp}(\tau U)}$, we have
$\Hom(Y_X', \tau U)=0$, and hence $\Hom(\PP_{U_X}, \PP_{Y_X'}[1])= 0$, by Lemma \ref{rigid-rigid}.
Hence there are maps $\PP_{U_X} \to \PP_{U_X' \amalg V_X'}$ and 
$\PP_{Y_X} \to \PP_{Y_X'}$ completing the above diagram in such a way that
there is a triangle (see \cite[Section 1.4]{nee})
$$\PP_{U_X' \amalg V_X'}[-1] \xrightarrow{g} \PP_{Y_X} \xrightarrow{h} \PP_{U_X \amalg Y_X'} \to  \PP_{U_X' \amalg V_X'}.$$
Now, let $k \colon \PP_{Y_X} \to \PP_T$ be a map with $T$ in $\T= {^{\perp}(\tau U \amalg \tau V)}$. Then $\Hom(T, \tau U \amalg \tau V) = 0$, and hence by Lemma \ref{rigid-rigid},
we have $\Hom(\PP_{U_X' \amalg V_X'}, \PP_T[1])= 0$. Hence we have $kg = 0$, so 
by exactness of $\Hom(\ ,\PP_T)$ we have that there is map $l \colon 
\PP_{U_X \amalg Y_X'} \to \PP_T$ such that $lh  =k$.
It then follows that the map
$H^0(k) \colon Y_X \to T$ factors through $H^0(l)$,
and by Lemma \ref{rigid-rigid}, it then follows that any map $Y_X \to T$ factors through
$Y_X \to U_X \amalg Y_X'$. This finishes the proof of the claim. 

\bigskip

\noindent (c):
We have $U\in \P({^{\perp}(\tau U)})$ by~\cite[Proposition 2.9]{air}, and
$Y_X\in \P({^{\perp}(\tau U)})$ by construction (see Definition~\ref{d:EUindecomposable}, Case I(b)).
Hence, in particular, $\Hom(U, \tau Y_X) = 0$, since
$\P({}^{\perp} \tau U)$ is $\tau$-rigid by~\cite[Theorem 2.10]{air}.
Then the assertion follows from Lemma \ref{al5}. 

\bigskip

\noindent (d):
We have that $f_U(Y_X)$ is in $\P(J(U))$, and hence $\tau_{J(U)} f_U(Y_X)= 0$.
Hence, the assertion follows from Lemma \ref{al5}.

\bigskip

\noindent (e): We have that the image of the map  $f^{J(U)}_{\overline{Q}}f_U$ is in $\overline{Q}^{\perp} 
\cap J(U) = \T'$, so in particular $f^{J(U)}_{\overline{Q}}f_U (Y_X)$ is in $\T'$.
By \cite[Lemma 5.6]{bm}), $f^{J(U)}_{\overline{Q}}f_U (Y_X')$ is indecomposable, so the map $f_{\overline{Q}}^{J(U)} f_U(\gamma)$ in (d) is left minimal.
The assertion follows.

\bigskip

\noindent (f): This follows from $f_U(Y_X') \in f_U(\T) = \T' = \overline{Q}^{\perp} 
\cap J(U)$.

\bigskip

\noindent (g): This now follows from (e) and (f).
\end{proof}

We have now proved that (\ref{eqv-1}) holds, and hence
(\ref{assoc}) holds in this case. 

\bigskip
 
\noindent {\bf Case I** (d):} 
We assume that $X$ is of the form $R[1]$, where $R$ is an
indecomposable module in $(\P(\Lambda) \cap {^{\perp}(U \amalg V)})$.

\bigskip

Note first that we have $\overline{\V} = \overline{Q}[1]$, where $\overline{Q} = f_U(Y_V)$
and where there is an exact sequence
$$Y_V \to U_V \to V \to 0,$$
where the first map is a minimal left $\add U$-approximation, and the second map
is a minimal right $\add U$-approximation.
  
We have that $\E_U(X) = f_U(Y_R)[1]$, where 
$R \xrightarrow{\beta} Y_R$ is a minimal left ${^\perp(\tau U)}$-approximation.
Furthermore, we have 
$$\E_{\overline{\V}}^{J(U)} \E_U(X) = \E_{\overline{Q}[1]}^{J(U)} \E_U(X) =
f_{\overline{Q}}^{J(U)} f_U(Y_R)[1].$$  

We have $f_U = f_{U \amalg V}$, by Lemma \ref{tf}.
We hence have that $\E_{U \amalg V}(X) = f_{U \amalg V}(Y'_R)[1] = f_{U}(Y'_R)[1]$, where 
$R \xrightarrow{\alpha} Y_R'$ is a left minimal ${}^\perp (\tau U\amalg \tau V)$-approximation. 
 
Hence, it will be sufficient to prove that 
\begin{equation}\label{eqv-2}
  f^{J(U)}_{\overline{Q}} f_U(Y_R) \simeq f_{U}(Y_R').
\end{equation}

The main steps in the proof are as follows.
 
\begin{claim}\label{c:2d} 
Let $\T = {^{\perp}(\tau U \amalg \tau V)}$ and let 
$\T' =  {\overline{Q}}^{\perp} \cap J(U)$. Then the following hold.
\begin{itemize}
\item[(a)]  We have $f^{J(U)}_{\overline{Q}} f_U(\T) =  f_U(\T) = \T'$.
\item[(b)] There is a map $Y_R \xrightarrow{a} Y_R'$ such that $f_U(a)$ is a left
$f_U(\T)$-approximation.
\item[(c)] We have $f^{J(U)}_{\overline{Q}}f_U(Y_R') \simeq f_U(Y_R')$.
\item[(d)] The map $f^{J(U)}_{\overline{Q}}f_U(Y_R) \xrightarrow{f^{J(U)}_{\overline{Q}}f_U(a)}
 f^{J(U)}_{\overline{Q}}f_U(Y_R')$ is a left $f^{J(U)}_{\overline{Q}} \T'=\T'$ -approximation.
\item[(e)]  The map $f^{J(U)}_{\overline{Q}}f_U(a)$ is an isomorphism.
 \item[(f)] We have $f^{J(U)}_{\overline{Q}}f_U(Y_R') \simeq f_U(Y_R')$.
 \item[(g)] We have $f^{J(U)}_{\overline{Q}} f_U(Y_R) \simeq f_{U}(Y_R')$.
\end{itemize}
\end{claim}
 
\begin{proof} 
(a): See Claim \ref{c:2c}(a).

\noindent (b): Since $\beta \colon R \to Y_R$ 
is a left ${^{\perp}(\tau U)}$-approximation, and $Y_R' \in \T \subseteq {^{\perp}(\tau U)}$, there is a map $a \colon  Y_R \to Y_R'$, such that $a \beta = \alpha$.

We claim that $f_U(a)$ is a left
$f_U(\T)$-approximation. Consider a map $y \colon f_U (Y_R) \to f_U (T)$, where 
$f_U (T)$ is in $f_U (\T)$.
Since $U, Y_R \in \P({^{\perp}(\tau U)})$, we have in particular that 
$\Hom(U, \tau Y_R)=0$. It then follows from Lemma \ref{al1} that $y = f_U(x)$
for some $x \colon Y_R \to T$.
Consider the diagram
$$
\xymatrix{
R  \ar@/^1.7pc/[rr]^{\alpha} \ar[r]^{\beta} \ar^{x\beta}[dr] & Y_R \ar^{a}[r] \ar^{x}[d] & Y_R' \ar@{.>}^{c}[dl] \\
& T &  
}
$$
where $c \colon Y_R' \to T$ such that $x\beta = ca\beta$ exists since
$\alpha = a\beta$ is a left $\T$-approximation. 
It follows that $(ca -x)\beta = 0$.
Applying $\Hom(\ ,T)$ to the right exact sequence 
$R \xrightarrow{\beta} Y_R \xrightarrow{l} U_R \to 0$  
gives the left exact sequence
$$0\rightarrow \Hom(U_R,T) \to \Hom(Y_R,T) \to \Hom(R,T).$$
Now $(ca -x)\beta = 0$ implies that there is a map  $n \colon U_R \to  T$,
such that $ca-x = nl$, and so $x = ca +nl$. Since $f_U(nl)= 0$,
this gives $y = f_U(x) = f_U(c) f_U(a)$. Hence we have that 
$f_U(a)$ is a left
$f_U(\T)$-approximation as claimed.

\bigskip

\noindent (c): This follows directly from (a), since $f_{\overline{Q}}^{J(U)}$
acts as the identity on $\overline{Q}^{\perp} \cap J(U)$.

\bigskip

\noindent (d,e,f,g): See Claim~\ref{c:2c}(d,e,f,g).
\end{proof}

We have now proved that (\ref{eqv-2}) holds, and hence that
equation~\eqref{assoc} holds in this case.  
  
\section{Proof of Theorem~\ref{main-asso}, Case III} 
\label{s:caseIII}
We have already dealt with Case II (see the end of Section~\ref{as}), so we must next deal with Case III.
We assume that $\U = P[1]$, where $P$ lies in $\P(\Lambda)$,
that $\V= V$ is a $\tau$-rigid module satisfying $\Hom(P,V)= 0$.

Then $\E_{\U}(\V)= \overline{\V} = \overline{V} = V$ 
is $\tau$-rigid in $J(U)= P^{\perp}$.

We need in this case to consider three possible cases for $X$:
\begin{itemize}
\item[(a)] $X$ lies in $\ind(\Lambda)$, $X$ does not lie in $\Gen V$,
$X \amalg V$ is $\tau$-rigid and $\Hom(P, X) = 0$.
\item[(b)] $X$ lies in $\ind(\Lambda)$, $X$ lies in $\Gen V$,
$X \amalg V$ is $\tau$-rigid and $\Hom(P, X) = 0$.
\item[(c)] $X$ is of the form $R[1]$ where $R$ lies in $\ind \P(\Lambda)$ and $\Hom(R,V) =0$. 
\end{itemize}

\bigskip

\noindent {\bf Case III (a):}
We assume that $X$ is an indecomposable $\tau$-rigid module
in $\module \Lambda$, that $X \amalg V$ is $\tau$-rigid,
that $\Hom(P, X) = 0$ and that $X$ does not lie in $\Gen V$. 

\bigskip

Then $\E_{\U}(X) = \E_{P[1]}(X) = X$, and $X$ is $\tau$-rigid in $J(\U) = P^{\perp}$.
We have 
$$\E_{\E_{\U}(\V)}^{J(\U)} \E_{\U}(X)= \E_{\overline{V}}^{J(\U)} \E_{\U}(X) = \E_{V}^{J(U)}(X) = f_V^{J(P[1])}(X) =
f_V^{P^{\perp}}(X).$$

We next compute $\E_{\U \amalg \V}(X)$.
We have $\E_{\V}(\U) = \E_{\V}(P[1]) = f_V (Y_P)[1]$, where $P \to Y_P$ is a left 
${^\perp(\tau V)}$-approximation.
We then obtain
$$\E_{\U \amalg \V}(X) = \E_{V \amalg P[1]}(X) = \E_{\E_{\V}(P[1])}^{J(V)} \E_V(X)
= \E_{f_V(Y_P)[1]}^{J(V)} \E_V(X) = \E_{f_V(Y_P)[1]}^{J(V)} f_V(X) = f_V(X),$$ 
where the second equality holds by definition. The last equation holds since
$\E_V(X \amalg P[1]) = f_V(X) \amalg f_V(Y_P)[1]$ is support $\tau$-rigid in $\C(J(U))$
by Theorem \ref{rigid-sums}, so $f_V(Y_P)$ is in $\P(J(U))$ with
$\Hom(f_V(Y_P), f_V (X)) = 0$.

We next claim that $f_V^{P^{\perp}}(X) = f_V(X)$.
For this, consider the canonical sequence of $X$ in $\module \Lambda$ with respect
to the torsion pair $(\Gen V, V^{\perp})$:
$$0 \to t_V(X) \to X \to f_V(X) \to 0.$$
Since $\Hom(P,X)= 0$ by assumption, and $P$ is projective, we also have 
$\Hom(P, f_V(X)) = 0$, and clearly also $\Hom(P, t_V(X)) =0$.
We have $t_V(X) \in \Gen V \cap P^{\perp} = \Gen_{P^{\perp}}V $
and $f_V(X) \in P^{\perp} \cap V^{\perp}$, so this sequence is also the canonical sequence 
of $X$ in $P^{\perp}$ with respect to the torsion pair $( \Gen_{P^{\perp}}V,
P^{\perp} \cap V^{\perp})$. Hence $f_V^{P^{\perp}}(X) = f_V(X)$ and
 it follows that 
$$\E_{\U \amalg \V}(X) = f_V(X) = f_V^{P^{\perp}}(X) = \E_{\overline{V}}^{J(U)} \E_{\U}(X),$$
and equation \eqref{assoc} holds also in this case.

\bigskip
    
\noindent {\bf Case III (b):}
We assume that $X\amalg V$ is a $\tau$-rigid module in $\module
\Lambda$ such that $X$ lies in $\Gen V$ and $\Hom(P,X) = 0$.

\bigskip

First note that $\E_{\U}(X) = \E_{P[1]}(X) = X$.
Consider the right exact sequence in $J(\U) = P^{\perp}$,
\begin{equation}\label{eqv-3}
Y_X^{P^{\perp}} \to V_X^{P^{\perp}} \to X \to 0,
\end{equation}
where the first map is a minimal left $\add V$-approximation in $P^{\perp}$, 
$Y_X\in \P({}^{\perp}(\tau V))$ and the
second map is a minimal right $\add V$-approximation in $P^{\perp}$.

We then have that $$\E_{\E_{\U}(\V)}^{J(\U)} \E_{\U} (X) = \E_{\overline{\V}}^{P^{\perp}} \E_{\U} (X) = \E_V^{P^{\perp}}(X) = f_V^{P^{\perp}} (Y_X^{P^{\perp}})[1].$$

We next compute 
$\E_{\U \amalg \V}(X)$.
First note that $\E_V(P[1]) = f_V(Y_P)[1]$, where $P \to Y_P$ is a minimal
left ${^{\perp}(\tau V)}$-approximation, and that 
$\E_V(X) = f_V(Y_X)[1]$, where there is an exact sequence
\begin{equation}\label{eqv-4}
Y_X \to V_X \to X \to 0 
\end{equation}
where the first map is a minimal left $\add V$-approximation, and the
second map is a minimal right $\add V$-approximation.
Then 
$$\E_{\U \amalg \V}(X) = \E_{P[1] \amalg V} = \E^{J(V)}_{\E_V(P[1])}\E_V(X)
= f_{f_V(Y_P)}^{J(V)} f_V(Y_X)[1].$$

Hence we need to prove that 
\begin{equation}\label{eqv5}
f_V^{P^{\perp}}(Y_X^{P^{\perp}}) \simeq  f_{f_V(Y_P)}^{J(V)}f_V(Y_X).
\end{equation}

We first make the following observation.

\begin{lemma}
Let $P$ be a projective module in $\module \Lambda$.
Then $f_P$ is a right exact functor from $\module \Lambda$ to
$P^{\perp}$, and $f_P$ sends projective modules to projective
modules in $P^{\perp}$.
\end{lemma}

\begin{proof}
Let $e$ be an idempotent such that $P \simeq \Lambda e$. 
We first note that $t_P(M) = \Lambda e M$, so $f_P(M) \simeq M/ \Lambda e M \simeq \Lambda/\Lambda e \Lambda 
\otimes_{\Lambda} M$. It follows that $f_P$ is right exact.

Moreover, since  $\Lambda/\Lambda e \Lambda 
\otimes_{\Lambda} \Lambda \simeq \Lambda/\Lambda e \Lambda$, and the tensor-functor is additive, we have that 
the indecomposable projective $\Lambda/\Lambda e \Lambda$-modules are
exactly $f_P(T)$ for $T$ indecomposable projective in $\module \Lambda$ with $T$ not a summand in $P$.
\end{proof}

We proceed to prove (\ref{eqv5}). The main steps in the proof
are as follows.

\begin{claim}
\begin{itemize}
\item[(a)] We have $f_P(Y_X) \simeq Y_X^{P^{\perp}}$. 
\item[(b)] The composition $Y_X  \xrightarrow{\alpha} f_P(Y_X)  \xrightarrow{\beta} f_{\overline{V}}^{P^{\perp}}f_P(Y_X)$ is a left $J(V) \cap P^{\perp}$-approximation.
\item[(c)] The composition $Y_X \xrightarrow{\gamma} f_V(Y_X) 
\xrightarrow{\phi} f_{f_V(Y_P)}^{J(V)} f_V(Y_X)$ is a left $J(V) \cap P^{\perp}$-approximation.
\item[(d)] We have $f_{\overline{V}}^{P^{\perp}}f_P(Y_X) \simeq f_{f_V(Y_P)}^{J(V)} f_V(Y_X)$.
\item[(e)] We have $f_V^{P^{\perp}}(Y_X^{P^{\perp}}) \simeq  f_{f_V(Y_P)}^{J(V)}f_V(Y_X)$.
\end{itemize}
\end{claim}

\begin{proof} \noindent (a):
Let $\T = {^{\perp}(\tau V)}$ and let 
$\T' = {^{\perp}(\tau_{P^{\perp}}} \overline{V}) \cap P^{\perp}$.
Then we have $f_P \T = \T'$ by Lemma \ref{al3}.
We have that $Y_X^{P^{\perp}}$ is in $P^{\perp} \cap {^{\perp}(\tau V)} =
 P^{\perp} \cap {^{\perp}(\tau \overline{V})}$.

Note that since $\Hom(P,X)= 0 = \Hom(P,V_X)$, we have $f_P (X) =X$ and $f_P(V_X) = V_X$.
Hence, applying $f_P$ to the right exact sequence (\ref{eqv-4})
we obtain the right exact sequence
\begin{equation}\label{eqv6}
f_P (Y_X) \xrightarrow{f_P(a)} V_X \to X \to 0.
\end{equation}
We claim that $f_P(a)$ is a minimal left $\overline{V}=V$-approximation in $P^{\perp}$.
Let $b' \colon f_P(Y_X) \to f_P(V') = V'$ be a map with $V' \in \add V \subseteq P^{\perp}$. By Lemma \ref{al1b}, there is a map $b \colon Y_X \to V'$ such that
$b' = f_P(b)$. Since $a$ is left $\add V$-approximation, there is a map 
$c \colon V_X \to V'$ such that $b =ca$. So $f_P(c) f_P(a) = f_P(b)$.
We have that $f_P(a)$ is minimal, since otherwise we would have $X$ in $\add V$.

Using now \cite[Proposition 5.6]{bm}, we have that $f_P(Y_X) \simeq Y_X^{P^{\perp}}$, and this concludes the proof of (a).

\bigskip

\noindent (b):
Let $Z$ be in $J(V) \cap P^{\perp}$ and let $g \colon Y_X \to Z$ be a map.
Since $\alpha$ is a left $P^{\perp}$-approximation and $Z$ is in $P^{\perp}$, there is a map $l \colon f_P(Y_X) \to Z$, such that $l \alpha = g$. Since $\beta$ is a left $P^{\perp} \cap V^{\perp}$-approximation and $Z$ is in $P^{\perp} \cap V^{\perp}$, there is a map $r \colon f_{\overline{V}}^{P^{\perp}}f_P(Y_X)
\to Z$, such that $l = r \beta$. Hence $g = l \alpha = r \beta \alpha$.
Since $Y_X$ is in ${^\perp(\tau V)}$, we have that also $f_{\overline{V}}^{P^{\perp}}f_P(Y_X)$ is in  ${^\perp(\tau V)}$, and hence
$f_{\overline{V}}^{P^{\perp}}f_P(Y_X)$ is in $P^{\perp} \cap V^{\perp} \cap 
{^\perp(\tau V)} = J(V) \cap P^{\perp}$. This 
proves the claim that $\beta \alpha$ is  a left $J(V) \cap P^{\perp}$-approximation.

\bigskip

\noindent (c):
Let $Z$ be in $J(V) \cap P^{\perp}$ and let $g \colon Y_X \to Z$ be a map.
Since $\gamma$ is a left $V^{\perp}$-approximation, and
$Z$ is in $V^{\perp}$, there is a map $s \colon f_V(Y_X) \to Z$ such that $g = s \gamma$.
Note that we have that  
$f_V(Y_P)$ is in $\P(J(V))$, and so 
$$(f_V(Y_P))^{\perp} \cap J(V) = J_{J(V)}(f_V(Y_P)[1]) = J_{J(V)}(\E_{V}(P[1])) = J(V \amalg P[1]) = J(V) \cap P^{\perp}.$$
The map $\phi$ is a left $(f_V(Y_P))^{\perp} \cap J(V) =  J(V) \cap P^{\perp}$-approximation.
Hence, there is a map $t \colon f_{f_V(Y_P)}^{J(V)} f_V(Y_X)  \to Z$ such that $s = t \phi$, and therefore $g = s \gamma = t \phi \gamma$.
This proves the claim that $\phi \gamma$ is  a left $J(V) \cap P^{\perp}$-approximation.

\bigskip

\noindent (d):
Since both $\beta \alpha$ and $\phi \gamma$ are epimorphisms, they are both minimal
left $J(V) \cap P^{\perp}$-approximations, and the claim follows.

\bigskip

\noindent (e): 
This follows directly from combining (a) and (d).
\end{proof}

Equation~\eqref{assoc} in this case now follows from \eqref{eqv5}.

\bigskip

\noindent {\bf Case III (c):}    
We assume that $X = R[1]$, where $R$ is an indecomposable
projective module in $\module \Lambda$ and $\Hom(R,V)  = 0$.

\bigskip

We then have $\E_{\U}(X) = f_P(R)[1]$, which is in $\P(P^{\perp})[1]$ and
we have that $$\E_{\V}^{J(\U)}\E_{\U}(X) = \E_V^{P^{\perp}}(f_P(R)[1]).$$
Note that $V = \overline{V}$ is $\tau$-rigid in $J(\U) = P^{\perp}$.
Therefore $\E_V^{P^{\perp}}(f_P(R)[1]) = f_V^{P^{\perp}}(Y_0)[1]$, where
 $f_P(R) \to Y_0$ is a minimal left ${^{\perp}}(\tau_{P^{\perp}}V) \cap P^{\perp}$-approximation.
 
We have that $\E_V(P[1]) = f_V(Y_P)[1]$, where $P \to Y_P$ is a minimal
left ${^{\perp}(\tau V)}$-approximation and similarly $\E_V(X) =  f_V(Y_R)[1]$, where $R \to Y_R$ is a minimal
left ${^{\perp}(\tau V)}$-approximation. It follows that
\begin{equation} \label{e:UVimage}
\E_{\U \amalg \V}(X) = \E_{P[1] \amalg V}(X) = \E_{\E_V(P[1])}^{J(V)} \E_V (X)=
 \E_{f_V(Y_P)[1]}^{J(V)} \E_V (X) = f_{f_V(Y_P)}^{J(V)} f_V(Y_R)[1].
\end{equation}

So it will be sufficient to prove that
\begin{equation} \label{e:IIIc}
f_V^{P^{\perp}}Y_0 \simeq f_{f_V(Y_P)}^{J(V)}f_V(Y_R).
\end{equation}
The main steps in the proof of this are as follows.

\begin{claim}
\begin{itemize}
\item[(a)] We have that $Y_0$ is a direct
summand of $f_P(Y_R)$.
\item[(b)] We have that $f_V^{P^{\perp}}(Y_0)$ is a
direct summand of $f_V^{P^{\perp}}f_P(Y_R)$.
\item[(c)] The composition $$Y_R \xrightarrow{\alpha} f_P(Y_R) \xrightarrow{\beta} f_V^{P^{\perp}} f_P(Y_R)$$ is a minimal left $J(V) \cap P^{\perp}$-approximation.
\item[(d)] The composition $$Y_R \xrightarrow{\gamma} f_V(Y_R) \xrightarrow{\phi} 
f_{f_V(Y_P)}^{J(V)} f_V(Y_R)$$ is a minimal left $J(V) \cap P^{\perp}$-approximation.
\item[(e)] We have $f_V^{P^{\perp}} f_P(Y_R) \simeq f_{f_V(Y_P)}^{J(V)} f_V(Y_R)$.
\item[(f)] We have $f_V^{P^{\perp}}Y_0 \simeq f_{f_V(Y_P)}^{J(P)}f_V(Y_R).$
\end{itemize}
\end{claim}

\begin{proof}
\noindent (a):
Let $\T = {^{\perp}(\tau V)}$ and
$\T' = {^{\perp}(\tau_{P^{\perp}} V)} \cap P^{\perp}$.
Note that the map $f_P(R) \to Y_0$ is a minimal left 
$\T'$-approximation so, for the claim,
it is sufficient to prove that
$f_P(R) \to f_P(Y_R)$ is a left $\T'$-approximation.
By Lemma~\ref{al3}, we have that $f_P(\T)=\T'$, so by
Lemma~\ref{al5} we have that $f_P(R) \to f_P(Y_R)$ is
a left $\T'$-approximation (noting that $\Hom(P,\tau R)=0$ as $R$ is projective), giving the claim.

\bigskip

\noindent (b): This follows directly from (a).

\bigskip

\noindent (c): Note first that since $Y_R$ is in ${^{\perp}(\tau V)}$, also
the factor module $f_V^{P^{\perp}} f_P(Y_R)$ is in ${^{\perp}(\tau V)}$.
This module also lies in $P^{\perp}\cap V^{\perp}$ by the definition of
$f_V^{P^{\perp}}$, so it lies in $J(V)\cap P^{\perp}$.

Let $Z$ be in $J(V) \cap P^{\perp}$, and consider a map $g \colon Y_R \to Z$.
Since $\alpha$ is a left $P^{\perp}$-approximation, and $Z$ is in $P^{\perp}$,
so there is a map $l \colon f_P(Y_R) \to Z$, such that $l \alpha = g$.
Since $\beta$ is a left $P^{\perp} \cap V^{\perp}$-approximation, and $Z$ is in 
$P^{\perp} \cap V^{\perp}$, there is a map $r \colon f_V^{P^{\perp}} f_P(Y_R) \to Z$,
such that $l =  r\beta$. Hence we have $g = l \alpha = r \beta \alpha$, and this proves that $\beta\alpha$ is left $J(V)\cap P^{\perp}$ approximation. Since this composition is an epimorphism,
it must also be left minimal, giving the claim.

\bigskip

\noindent (d): First note that $f_V(Y_P)$ is in $\P(J(V))$, and hence
$(f_V(Y_P))^{\perp} \cap J(V) = J_{J(V)}(f_V(Y_P)[1]) = J_{J(V)}(\E_V(P[1])) = 
J(V \amalg P[1]) = J(V) \cap P^{\perp}$. 
Let $g \colon Y_R \to Z$ be a map, with $Z$ in $J(V) \cap P^{\perp}$.
Since $\gamma$ is a $V^{\perp}$-approximation, and $Z$ is in $V^{\perp}$, there is 
a map $t \colon f_V(Y_R) \to Z$ such that $g = t \gamma$. Since $\phi$ is a left
$(f_V(Y_P))^{\perp} \cap J(V) = J(V) \cap P^{\perp}$-approximation, there
is a map $u \colon f_{f_V(Y_P)}^{J(V)} f_V(Y_R) \to Z$, such that $t = u \phi$.
It follows that $g = t \gamma = u \phi \gamma$.
This proves that $\phi\gamma$ is left $J(V)\cap P^{\perp}$ approximation. Since this composition is an epimorphism,
it must also be left minimal, giving the claim.

\bigskip

\noindent (e): This follows directly from (c) and (d).
%, noting that both 
%$f_V^{P^{\perp}} f_P(Y_R)$ and $f_{f_V(Y_P)}^{J(V)} f_V(Y_R)$ are %indecomposable
%by \cite[Proposition 4.7]{bm}.

\bigskip

\noindent (f):
Note that $f_V^{P^{\perp}}(Y_0)$ is indecomposable
by~\cite[Proposition 5.6]{bm} (see the definition of $Y_0$ above).
It is a direct summand of $f_{f_V(Y_P)}^{J(V)} f_V(Y_R)$
which is indecomposable by~\eqref{e:UVimage}
and~\cite[Proposition 5.6]{bm}. The claim follows.
%This follows directly from combining (b) and (e).
\end{proof}

We have proved that~\eqref{e:IIIc} holds, and~\eqref{assoc}
in this case now follows.

\section{Proof of Theorem~\ref{main-asso}, Case IV}
\label{s:caseIV}
We assume that $\U= P[1]$ and $\V=Q[1]$, where $P,Q$ lie in $\P(\Lambda)$.

We set $\overline{V} = \overline{Q}[1]$. Then
$\overline{Q}=f_P Q$ lies in
$\P(J(\U))= \P(P^{\perp})$.

We need in this case to consider two possible cases for $X$:
\begin{itemize}
\item[(a)] $X$ is $\tau$-rigid, $X$ lies in $\ind(\Lambda)$, and
$\Hom(P \amalg Q, X) = 0$.
\item[(b)] $X$ lies in $\ind \P(\Lambda)[1]$.
\end{itemize}

\bigskip

\noindent {\bf Case IV (a):}
We assume that $X$ is an indecomposable $\tau$-rigid module with $\Hom(P \amalg Q, X) = 0$.

\bigskip

We have $\E_{\U}(X) = \E_{P[1]}(X) = X$, and then $\E_{\V}^{J(\U)} \E_{\U}(X)
= \E^{P^{\perp}}_{\overline{Q}[1]} X = X$, since $\Hom(\overline{Q}, X) = 0$ by 
Theorem \ref{rigid-sums}.

We also have $\E_{\U \amalg \V}(X) = \E_{P[1] \amalg Q[1]}(X) = X$, so the claim that equation~\eqref{assoc} holds follows
also in this case.

\bigskip

\noindent {\bf Case IV (b):}
We assume that $X$ is of the form $R[1]$, where $R$ is an indecomposable module in $\P(\Lambda)$.

\bigskip

We then have that 
$$\E_{\V}^{J(\U)} \E_{\U}(X) = \E_{\overline{Q}[1]}^{P^{\perp}} \E_{P[1]}(X)=
\E_{\overline{Q}[1]}^{P^{\perp}} (f_P(R)[1]) = f_Q^{P^{\perp}}f_P(R)[1].$$

On the other hand, we have $\E_{\U \amalg \V}(X) = \E_{P[1] \amalg Q[1]}(X)=
f_{P \amalg Q}(R)[1]$.

So, it is sufficient to prove that 
$f_Q^{P^{\perp}} f_P(R) \simeq f_{P \amalg Q}(R)$.
The main steps in the proof are as follows.

\begin{claim}
\begin{itemize}
\item[(a)] We have $P^{\perp} \cap \overline{Q}^{\perp} = P^{\perp} \cap Q^{\perp}$.
\item[(b)] The composition $R \xrightarrow{\alpha} f_P(R) \xrightarrow{\beta} f_{\overline{Q}}^{P^{\perp}}f_P(R)$ is a $(P \amalg Q)^{\perp}$-approximation.
\item[(c)] We have $f_Q^{P^{\perp}}f_P(R) \simeq f_{P \amalg Q}(R)$.
\end{itemize}
\end{claim}

\begin{proof}
\noindent (a): 
Note first that, for any module $M$, $\Hom(P,M) = 0$ implies $\Hom(\Gen P, M) = 0$. 
Suppose $M$ lies in $P^{\perp}\cap \overline{Q}^{\perp}$.
We apply $\Hom(\ ,M)$ to the canonical sequence
$$0 \to t_P(Q) \to Q \to \overline{Q} \to 0$$
for $Q$. We have $\Hom(\overline{Q}, M) = 0$ and also $\Hom(t_P(Q), M) = 0$, since $t_P(Q)$ is in $\Gen P$. Hence $\Hom(Q,M) =0$ and we have
$P^{\perp} \cap \overline{Q}^{\perp} \subseteq P^{\perp} \cap Q^{\perp}$.
The reverse inclusion follows immediately from the fact that $\overline{Q}$ is a factor of $Q$.

\bigskip 

\noindent (b): By (a), we have that $f_{\overline{Q}}^{P^{\perp}}(f_P(R))$ is in $P^{\perp}\cap \overline{Q}^{\perp}=P^{\perp} \cap Q^{\perp}$. 
Consider a map $g \colon R \to Z$ with $Z$ in $(P \amalg Q)^{\perp}$.
Since $\alpha$ is a left $P^{\perp}$-approximation and $Z$ is in $P^{\perp}$, there is a map $t \colon f_P(R) \to Z$
such that $g = t \alpha$. Since $f_P(R) \to f_{\overline{Q}}^{P^{\perp}}f_P(R)$ is
a left  $P^{\perp} \cap \overline{Q}^{\perp} = (P \amalg Q)^{\perp}$-approximation
and $Z$ is in $P^{\perp} \cap Q^{\perp}$,
there is a map $u  \colon f_{\overline{Q}}^{P^{\perp}}f_P(R) \to Z$ such that $t =u\beta$.
We then have $g = t \alpha =  u\beta \alpha$. This proves the claim.

\bigskip 

\noindent (c): This follows directly from (b), noting that both $f_{\overline{Q}}^{P^{\perp}}f_P(R)$
and $f_{P \amalg Q}(R)$ are indecomposable (since they are factors of the
indecomposable projective module $R$).
\end{proof}

This finishes the proof that (\ref{assoc}) holds in this case.

\section{End of the proof of Theorem~\ref{main-asso}: Mixed case}
\label{s:mixed}
We have now proved that~\ref{assoc} holds for all of the
cases I-IV. It remains to deal with the mixed cases, where
we have support $\tau$-rigid objects
$\U = U \amalg P[1]$ and $\V = V \amalg Q[1]$ in $\C(\Lambda)$, with no common direct summands, but where we allow indecomposable
direct summands of $\U$ and $\V$ to lie both in $\module \Lambda$ and in $\P(\Lambda)[1]$.

Let us summarize the formulas we need to proceed.
By Cases I-IV, we have that the formulas
\begin{equation}\label{pure}
\E^{J(\U)}_{\E_{\U}(\V)} \E_{\U} = \E_{\U \amalg \V} = \E^{J(\V)}_{\E_{\V}(\U)} \E_{\V} 
\end{equation}
hold when we have both of the following:
\begin{itemize}
\item $U = 0$ or $P=0$, and
\item $V = 0$ or $Q=0$.
\end{itemize}

Note that a particular case is when $\U = U$ and $\V = Q[1]$,
where $U$ lies in $\module\Lambda$ and $Q$ lies in $\P(\Lambda)$.
We therefore have
\begin{equation}\label{caseup}
\E_{U \amalg Q[1]} = \E^{J(U)}_{\E_U(P[1])} \E_U.
\end{equation}

Recall also from Section \ref{compo} that we have 
\begin{equation}\label{compos} 
J_{J(\U)}(\E_{\U}(\V)) = J(\U \amalg \V),
\end{equation}
for any pair of support $\tau$-rigid objects $\U, \V$ in $\C(\Lambda)$.

\bigskip

\noindent {\bf Case A:} We first discuss the case with $P = 0$, that is $\U = U \neq 0$, while $\V= V \amalg Q[1]$
is arbitrary. We work by induction on $n = r(\module \Lambda)$. We then have
\begin{align}
\nonumber \E_{\U \amalg \V} &= \E_{U \amalg V \amalg Q[1]} \\[5pt]
\label{en1} &= \E^{J(U \amalg V)}_{\E_{U \amalg V}(Q[1])} \E_{U \amalg V} \\[5pt]
\label{en2}&= \E^{J(U \amalg V)}_{\E_{U \amalg V}(Q[1])} \E_{\E_U(V)}^{J(U)} \E_U \\[5pt]
\label{en3}&= \E^{J_{J(U)}(\E_U(V))}_{\E_{U \amalg V}(Q[1])} \E_{\E_U(V)}^{J(U)} \E_U \\[5pt]
\label{en4}&=  \E^{J_{J(U)}(\E_U(V))}_{\E_{\E_U(V)}^{J(U)} \E_U(Q[1])} \E_{\E_U(V)}^{J(U)} \E_U \\[5pt]
\label{en5} &=  \E^{J(U)}_{\E_U(V) \amalg \E_U(Q[1])} \E_U \\[5pt]
\nonumber  &=  \E^{J(U)}_{\E_U(V \amalg Q[1])} \E_U  \\[5pt] 
\nonumber  &=  \E^{J(\U)}_{\E_{\U}(\V)} \E_{\U} 
\end{align}
where equation~\eqref{en1} follows from (\ref{caseup}), while~\eqref{en2} and (\ref{en4}) follows from (\ref{pure}) and (\ref{en3}) from (\ref{compos}). Furthermore, equation~\eqref{en5}
follows from the induction assumption, since $r(J(U)) < n$.
This concludes the proof of the case with $P = 0$, i.e. $\U = U$.

\bigskip 

\noindent {\bf Case B:}
We next discuss the case with $U=0$, that is $\U = P[1] \neq 0$, while $\V= V \amalg Q[1]$
is arbitrary. We also assume $V \neq 0$, note that we have already dealt with the
case $V = 0$ (this is Case IV).
We then have:
\begin{align}
\nonumber \E^{J(\U)}_{\E_{\U}(\V)} \E_{\U} & = \E^{J(P[1])}_{\E_{P[1]}(V \amalg Q[1])} \E_{P[1]} \\[5pt]
\nonumber & = \E^{J(P[1])}_{\E_{P[1]}(V) \amalg \E_{P[1]}(Q[1])} \E_{P[1]} \\[5pt]
\label{lik1}  & = \E^{J_{J(P[1])}(\E_{P[1]}(V))}_{\E_{\E_{P[1]}(V)}^{J(P[1])} \E_{P[1]}(Q[1])} \E_{\E_{P[1]}(V)}^{J(P[1])} \E_{P[1]} \\[5pt]
\label{lik2} & = \E^{J_{J(P[1])}(\E_{P[1]}(V))}_{\E_{P[1] \amalg V}(Q[1])} \E_{\E_{P[1]}(V)}^{J(P[1])} \E_{P[1]} \\[5pt]
\label{lik3} & = \E^{J(P[1] \amalg V)}_{\E_{P[1] \amalg V}(Q[1])} \E_{\E_{P[1]}(V)}^{J(P[1])} \E_{P[1]} \\[5pt]
\label{lik4} & = \E^{J(P[1] \amalg V)}_{\E_{P[1] \amalg V}(Q[1])} \E_{P[1] \amalg V} \\[5pt]
\label{lik5} & = \E^{J(P[1] \amalg V)}_{\E_{P[1] \amalg V}(Q[1])} \E_{\E_{V}(P[1])}^{J(V)} \E_{V} \\[5pt]
\label{lik6} & = \E^{J(P[1] \amalg V)}_{\E^{J(V)}_{\E_V(P[1])} \E_V(Q[1])}  \E_{\E_{V}(P[1])}^{J(V)} \E_{V}  \\[5pt]
\label{lik7} & = \E^{J_{J(V)}(\E_V(P[1]))}_{\E^{J(V)}_{\E_V(P[1])} \E_V(Q[1])}  \E_{\E_{V}(P[1])}^{J(V)} \E_{V}  \\[5pt] 
\label{lik8} & = \E^{J(V)}_{\E_V(P[1] \amalg Q[1])} \E_{V} \\[5pt]
\label{lik9}& = \E_{V \amalg P[1] \amalg Q[1]} \\[5pt]
\nonumber & = \E_{\U \amalg \V}
\end{align}
where for (\ref{lik1}), we use that (\ref{assoc}) holds in $J(P[1])$ by induction.
For (\ref{lik8}), we use that the (\ref{assoc}) holds in $J(V)$ by induction.
For (\ref{lik2}) and (\ref{lik4}) we apply (\ref{pure}), for (\ref{lik5}), (\ref{lik6})
and (\ref{lik9})  we apply (\ref{caseup}), 
while for (\ref{lik3}) and  (\ref{lik7}) we apply~\eqref{compos}. 

\bigskip

\noindent {\bf The general case: } We now discuss the general case with $\U = U \amalg P[1]$ and $\V = V \amalg Q[1]$.

We then have
\begin{align}
\nonumber \E^{J(\U)}_{\E_{\U}(\V)} \E_{\U} & = \E^{J(U \amalg P[1])}_{\E_{U \amalg P[1]}(V \amalg Q[1])} \E_{U \amalg P[1]} \\[5pt]
\label{li1} & = \E^{J(U \amalg P[1])}_{\E_{U \amalg P[1]}(V \amalg Q[1])} \E^{J(U)}_{\E_U(P[1])} \E_U
\\[5pt]
\label{li2} & = \E_{\E_{\E_U(P[1])}^{J(U)}\E_U(V \amalg Q[1])}^{J(U \amalg P[1])}  \E^{J(U)}_{\E_U(P[1])} \E_U \\[5pt]
\label{li3} & = \E_{\E_{\E_U(P[1])}^{J(U)}\E_U(V \amalg Q[1])}^{J_{J(U)}(\E_U(P[1]))}  \E^{J(U)}_{\E_U(P[1])} \E_U \\[5pt]
\label{li4} & = \E^{J(U)}_{\E_U(P[1]) \amalg \E_U(V \amalg Q[1])} \E_U \\[5pt]
\nonumber & = \E^{J(U)}_{\E_U(P[1] \amalg V \amalg Q[1])} \E_U \\[5pt]
\nonumber & = \E^{J(U)}_{\E_U(V \amalg P[1] \amalg Q[1])} \E_U \\[5pt]
\label{li5}& = \E_{U \amalg P[1] \amalg V \amalg Q[1]} \\[5pt]
\nonumber & = \E_{\U \amalg \V}
\end{align}
where (\ref{li1}) and (\ref{li2}) hold by (\ref{caseup}), while (\ref{li3}) holds
by (\ref{compos}). For (\ref{li4}) we note
that $\E_U(P[1])$ is in $\P(J(U))[1]$, so that we are in the situation of Case B in $J(U)$.
For equation (\ref{li5}), we apply Case A.

This concludes the proof of Theorem~\ref{main-asso}.$\Box$

\section{Irreducible morphisms in \texorpdfstring{$\mathfrak{W}_{\Lambda}$}{W Lambda}}
\label{s:irreducible}
%\section{Morphisms to wide subcategories of corank $1$}

In this section we prove the following Theorem.

\begin{theorem}\label{main-irr}
Let $\Lambda$ be a $\tau$-tilting finite algebra, and let $\W' \subseteq \W$
be wide subcategories of $\module \Lambda$, where $r(\W) - r(\W') = 1$ (i.e $\W'$ is of corank $1$ in
$\W$) . Then exactly one of the following occurs:
\begin{itemize}
\item[(a)] There is exactly one morphism in $\mathfrak{W}_{\Lambda}$ from $\W$ to $\W'$ and $\W' = J_{\W}(U)$, where $U$ is an indecomposable $\tau$-rigid 
module which is non-projective in $\W$. 
\item[(b)] There are exactly two morphisms in $\mathfrak{W}_{\Lambda}$
from $\W$ to $\W'$ and $\W' = J_{\W}(P) = J_{\W}(P[1])$, where $P$ is an indecomposable module which is projective in $\W$.
\end{itemize} 
\end{theorem}
 
The main step in the proof is to show that if $U$ and $V$ are indecomposable $\tau$-rigid $\Lambda$-modules satisfying $J(U)=J(V)$, then $U=V$.

\begin{definition}
A morphism $g$ in $\mathfrak{W}_{\Lambda}$ is said to be {\em irreducible} if,
whenever $g$ is expressed as a composition $g_1 \circ g_2$, we
have that either $g_1$ or $g_2$ is an identity map.
\end{definition}

\begin{lemma} \label{l:irreducible}
Let $\W$ be a wide subcategory of $\module \Lambda$ and
let $\V$ be a support $\tau$-rigid object in $\C(\W)$. Then the following are equivalent.
\begin{itemize}
\item[(a)] The morphism $g = g_{\V}^{\W} \colon \W \to J_{\W}(\V)$ is
irreducible. 
\item[(b)] The object $\V$ is indecomposable.
\item[(c)] The subcategory $J_{\W}(\V)$ is of corank 1 in $\W$.
\end{itemize}
\end{lemma}

\begin{proof}
Suppose first that $\V$ is indecomposable, and that $g=g_1\circ g_2$
for maps $g_1$ and $g_2$ in $\mathfrak{W}$. Then we have $g_1=g_{\U_1}^{\W_1}$ and $g_2=g_{\U_2}^{\W_2}$ where $\U_1$ is
a support $\tau$-rigid object in $\C(\W_1)$, $\U_2$ is a support $\tau$-rigid
object in $\C(\W_2)$ and $J_{\W_2}(U_2)=\W_1$. The composition is:
$$g_{\V}^{\W}=g_{\U_1}^{\W_1}\circ g_{\U_2}^{\W_2}=g^{\W_2}_{\F_{\U_2}(\U_1)\amalg \U_2}.$$
Hence $\W_2=\W$ and $\V=\F_{\U_2}(\U_1)\amalg \U_2$. Since $\V$ is indecomposable, we have $\U_2=0$ or $\F_{\U_2}(\U_1)=0$. 
So $\U_1=0$ or $\U_2=0$, and $g_1$ or $g_2$ is an identity map.
It follows that $g$ is irreducible. This proves that (b) implies (a).

If $\V$ is decomposable, it can be written in
the form $\V=\V_1\amalg \V_2$ where $\V_1$ and $\V_2$ are non-zero
support $\tau$-rigid objects in $\C(\W)$. Then we have:
$$g=g_{\V_1\amalg \V_2}^{\W}=g_{\E_{\V_2}(\V_1)}^{\W_1}\circ g_{\V_2}^{\W_2},$$
where $\W_2=\W$ and $\W_1=J_{\W_2}(V_2)$. Since $\V_1$ and $\E_{\V_2}(\V_1)$ are non-zero, $g_{\E_{\V_2}(\V_1)}^{\W_1}$ and $g_{\V_2}^{\W_2}$ are not
identity maps, so $g$ is not irreducible. This proves that (a) implies (b).

We have that (b) and (c) are equivalent by Proposition \ref{prop-wide}.
\end{proof}

Recall that for any $\tau$-rigid $\Lambda$-module $U$ there is a unique basic module $B_U$, known as the \emph{Bongartz complement} of $U$, such that $\add(U\amalg B_U)$ is a $\tau$-tilting module and
$\add(U\amalg B_U)=\P({}^{\perp} \tau U)$.
We also recall that a $\Lambda$ module $M$ is said to be \emph{Gen-minimal}
if, whenever $M=M'\oplus M''$, we have $M''\not\in \Gen M'$ (see e.g.~\cite[VI.6]{assemsimsonskowronski}). We recall the following:

\begin{lemma} \cite[Lemma 2.8]{ingallsthomas}
Let $\Lambda$ be an algebra, and let $\T$ be a finitely generated torsion class in $\module \Lambda$. Then $\T$ has a unique Gen-minimal
generator, $\T_{\text{min}}$, consisting of the direct sum of the indecomposable split projective objects in $\T$.
\end{lemma}

If $T$ is a support $\tau$-tilting module, then we denote the unique Gen-minimal
generator of $\Gen T$ by $T_{\text{s}}$. Note that $T$ is an additive
generator for $\P(\T)$ by~\cite[Thm. 2.7]{air}, so $T_{\text{s}}$ is a direct
summand of $T$, and we write $T_{\text{ns}}$ for a complement, the direct
sum of the non-split projective objects in $\Gen T$.

If $Z$ is a minimal direct summand of $T$ such
that $\Gen Z=\Gen T$ then $T_{\text{s}}\in \Gen =Z$, so $T_{\text{s}}$ is
a direct summand of $Z$ since it is split projective. Since
$\Gen T_{\text{s}}=\Gen T$, we must have $T_{\text{s}}=Z$.
In the light of this discussion, we also recall the following:

\begin{theorem} \cite[Thm.\ 3.34]{dirrt} \label{t:dirrtbijection}
Let $\Lambda$ be a $\tau$-tilting finite algebra.
Then there is a bijection between the set of $\tau$-tilting pairs
in $\module \Lambda$ and the set of wide subcategories of $\module\Lambda$
given by mapping a $\tau$-tilting pair $(T,P)$ to $\mathsf{W}(T,P)=
J(T_{ns})\cap P^{\perp}$.
\end{theorem}

\begin{lemma} \label{l:Bongartzsplit}
Let $\Lambda$ be a $\tau$-tilting finite algebra.
Let $U$ be a non-projective $\tau$-rigid module in $\module\Lambda$. Let $B_U$ be the Bongartz complement of $U$, and let $T_U=U\amalg B_U$. Then
$(T_U)_{\text{s}}=B_U$ and $(T_U)_{\text{ns}}=U$.
\end{lemma}
\begin{proof}
By the definition of Bongartz complement, we have that
$\add(T_U)=\P({}^{\perp} (\tau U))$.
By~\cite[Lemma 4.12]{bm}, the indecomposable direct summands of $B_U$ are 
split projective in ${}^{\perp} (\tau U)$.
Suppose that $U$ was also split projective in ${}^\perp \tau U$. Then we would
have $(T_U)_{\text{ns}}=0$ and therefore $\W(T_U,0)=\module \Lambda$
in Theorem~\ref{t:dirrtbijection}. But $\W(P,0)=\module \Lambda$, where $P$
is an additive generator for $\P(\Lambda)$, so $T_U=P$ by
Theorem~\ref{t:dirrtbijection}, and $U$ is projective, giving a contradiction.
Hence $U$ is not split projective in ${}^{\perp} \tau U$ and
we are done.
\end{proof}

\begin{proposition} \label{p:JUJV}
Let $\Lambda$ be a $\tau$-tilting finite algebra.
Let $U$ and $V$ be indecomposable $\tau$-rigid $\Lambda$-modules and suppose that $J(U)=J(V)$. Then $U=V$.
\end{proposition}
\begin{proof}
Let $B_U$ (respectively, $B_V$) be the Bongartz complement of $U$ (respectively, $V$), and set $T_U=U\amalg B_U$ and $T_V=V\amalg B_V$. Then, since
$T_U$ and $T_V$ are $\tau$-tilting modules, we have that
$(T_U,0)$ and $(T_V,0)$ are $\tau$-tilting pairs. We have $\W(T_U,0)=
J((T_U)_{\text{ns}})=J(U)$ by Lemma~\ref{l:Bongartzsplit}, and similarly
$\W(T_V,0)=J(V)$. So, by Theorem~\ref{t:dirrtbijection}, $U=V$.
\end{proof}

We now finish the proof of Theorem \ref{main-irr}.
\begin{proof}[Proof of Theorem \ref{main-irr}:]
By Lemma~\ref{l:irreducible} and Proposition~\ref{tau-finite}, we have
$\W' = J_{\W}(\U)$ where $\U$ is either an indecomposable
$\tau$-rigid module or $\U= P[1]$ for an indecomposable module $P$ which is projective in $\W$. The result now follows from Proposition~\ref{p:JUJV}
and the fact that $J_{\W}(P) =J_{\W}(P[1])$.
\end{proof}

\section{Morphisms in \texorpdfstring{$\mathfrak{W}_{\Lambda}$}{W Lambda} and signed \texorpdfstring{$\tau$}{tau}-exceptional sequences}
\label{s:factorization}
The notion of signed $\tau$-exceptional sequence was introduced in \cite{bm}. Such sequences can be interpreted as factorizations of morphisms in the category $\W_{\Lambda}$. Our aim in this section is to make a precise version of this statement. 

Recall from \cite{bm} that an object $M\amalg P[1]$ in
$\C(\Lambda)$ is said to be support \emph{$\tau$-rigid} if $M$ is
a $\tau$-rigid module in $\module\Lambda$, $P$ lies in
$\P(\Lambda)$ and $\Hom(P,M)=0$. Furthermore, a sequence 
\begin{equation}\label{tau-seq}
\SS = (\U_1, \U_2, \dots, \U_t)
\end{equation} 
of indecomposable objects in $\C(\Lambda)$ is said to be a {\em signed $\tau$-exceptional sequence} if
$\U_t$ is support $\tau$-rigid in $\C(\Lambda)$ and the subsequence
$(\U_1, \U_2, \dots, \U_{t-1})$ is a signed $\tau$-exceptional sequence in $J(\U_t)$. 

\begin{theorem} \cite[Thm.\ 5.4]{bm}\label{t:bmbijection}
For each $t\in \{1,\ldots ,n\}$ there is a bijection $\varphi_t$ from the set of signed $\tau$-exceptional sequences of length $t$ in $\C(\Lambda)$ to the set of ordered support $\tau$-rigid objects of length $t$ in $\C(\Lambda)$.
\end{theorem}

We have the following, noting that if $\Lambda$ is $\tau$-tilting finite
then every wide subcategory of $\module\Lambda$ is equivalent to a module
category, by Proposition \ref{tau-finite}.

\begin{corollary} \label{c:bmbijection}
Suppose that $\Lambda$ is $\tau$-tilting finite, and let
$\W$ be a wide subcategory of $\module\Lambda$.
Then for each $t\in \{1,\ldots ,n\}$ there is a bijection $\varphi^{\W}_t$ between the set of set of signed $\tau$-exceptional sequences of length $t$ in $\W$ and the set of ordered support $\tau$-rigid objects of length $t$ in $\C(\W)$.
\end{corollary}

Recall now the following fact from \cite[Remark 5.12]{bm}.

\begin{proposition}\label{p:phit}\cite{bm}
Assume that $\Lambda$ is $\tau$-tilting finite.
Let $\W$ be a wide subcategory of $\module\Lambda$.
Then the bijection $\varphi_t^{\W}$ in Corollary~\ref{c:bmbijection} is
given by
$$(\U_1,\ldots ,\U_t)\mapsto 
(\F_{\U_t}^{\W_t}\cdots \F_{\U_2}^{\W_2}(\U_1),
\F_{\U_t}^{\W_t}\cdots \F_{\U_3}^{\W_3}(\U_2),
\dots ,\U_t)$$
where $\W_t=\W$ and $\W_i=J_{\W_{i+1}}(\U_{i+1})$ for all $i$.
\end{proposition}
%\begin{proof}
%We recall the description of the bijection
%from~\cite[Proof of Thm.\ 5.4]{bm}.
%For $t=1$, the bijection $\varphi_1^{\W}$
%is equal to the identity map, so the
%Proposition holds for this case.
%For $t>1$, the bijection $\varphi_t^{\W}$ is given by
%$$(\U_1,\ldots ,\U_t)\mapsto (\F^{\W_t}_{\U_t}(\U'_1),
%\ldots ,\F_{\U_t}^{\W_t}(U'_{t-1}),U_t),$$
%where
%$$(\U'_1,\ldots ,\U'_{t-1})=\varphi^{\W_{t-1}}_{t-1}(\U_1,\ldots ,\U_{t-1}).$$
%
%We prove the result by induction on $t$. For $t=1$, the result
%is clear, so suppose that the result holds for $t-1$.
%Then $\varphi^{\W}_{t}$ maps $(U_1,\ldots ,U_t)$ to
%\begin{equation}
%\begin{split}
%(\F^{\W_t}_{\U_t}(\U'_1),& \ldots ,\F_{\U_t}(\U'_{t-1}),\U_t)= \\
%&(\F^{\W_t}_{\U_t}\F^{\W_{t-1}}_{\U_{t-1}}\cdots \F^{\W_2}_{\U_2}(\U_1),
%\F^{\W_t}_{\U_t}\F^{\W_{t-1}}_{\U_{t-1}}\cdots \F^{\W_3}_{\U_3}(\U_2),
%\cdots ,\F^{\W_t}_{\U_t}(\U_{t-1}), \U_t),
%\end{split}
%\end{equation}
%so the result holds for $t$.
%\end{proof}

To prepare for our main results in this section, we now state and prove the following three lemmas.
\begin{lemma} \label{l:welldefined}
Let $\W$ be a wide subcategory of $\module\Lambda$, and let $\U_1,\ldots ,\U_t$ be
indecomposable objects in $\C(\W)$.
Then the following are equivalent.
\begin{itemize}
\item[(a)] The sequence $(\U_1,\ldots ,\U_t)$ is a signed
$\tau$-exceptional sequence in $\W$;
\item[(b)] There are wide subcategories $\W_1,\ldots ,\W_t$ of
$\module \Lambda$ with $\W_t=\W$, and maps $g_{\U_i}^{\W_i}$ for
$i=1,\ldots ,t$, such that the composition
$g_{\U_1}^{\W_1}\cdots g_{\U_t}^{\W_t}$
is well-defined in $\mathfrak{W}$.
\end{itemize}
\end{lemma}
\begin{proof}
We prove that (a) implies (b) by induction on $t$.
If $t=1$ then $\U_1$ is support $\tau$-rigid in $\C(\W)$, so
there is a corresponding map $g_{\U_1}^{\W_1}$, taking $\W_1=\W$,
and the result holds for this case. Suppose the result holds for
$t-1$, and let $(\U_1,\ldots ,\U_t)$ be a signed $\tau$-exceptional
sequence in $\W$ of length $t$. Then $(\U_1,\ldots ,\U_{t-1})$ is
a signed $\tau$-exceptional sequence of length $t-1$ in $J(\U_t)$.
By the induction hypothesis, there are wide subcategories
$\W_1,\ldots ,\W_{t-1}$ of $\module\Lambda$ with $\W_{t-1}=J_{\W}(\U_t)$,
and maps $g_{\U_i}^{\W_i}$ for $i=1,\ldots ,t-1$, such that
the composition
$g_{\U_1}^{\W_1}\cdots g_{\U_{t-1}}^{\W_{t-1}}$
is well-defined. Since $\U_t$ is support $\tau$-rigid in $\C(\W)$, there is a
map $g_{\U_t}^{\W}:\W\rightarrow J_{\W}(\U_t)$ in $\mathfrak{W}$.
The result follows, taking $\W_t=\W$.

We prove that (b) implies (a) by induction on $t$. For $t=1$
the result is clear, so suppose that the result holds for $t-1$,
and let $\W_i$ and $g_{\U_i}^{\W_i}$ be as in (b).
Since the composition
$g_{\U_1}^{\W_1}\cdots g_{\U_{t-1}}^{\W_{t-1}}$
is well-defined, $(\U_1,\ldots ,\U_{t-1})$ is a signed $\tau$-exceptional sequence in $\W_{t-1}$ by the induction hypothesis.
Since $g_{\U_t}^{\W_t}=g_{\U_t}^{\W}$ is a map, $\U_t$ is
support $\tau$-rigid in $\C(\W)$, and since the composition
$g_{\U_1}^{\W_1}\cdots g_{\U_t}^{\W_t}$
is well-defined, we have $J_{\W_t}(\U_t)=\W_{t-1}$, giving (a).
\end{proof}

Let $\W$ be a wide subcategory of $\module\Lambda$. For a
signed $\tau$-exceptional sequence $\U_1,\ldots ,\U_t$ in $\W$, we
denote by $\overline{\varphi}_t^{\W}(\U_1,\ldots ,\U_t)$ the direct sum of the entries in $\varphi_t^{\W}(\U_1,\ldots ,\U_t)$.

\begin{lemma} \label{l:composition}
Let $\W$ be a wide subcategory of $\module\Lambda$, and suppose
that the sequence $(\U_1,\ldots ,\U_t)$ is a signed
$\tau$-exceptional sequence in $\W$.
Set $\W_t=\W$ and $\W_i=J_{\W_{i+1}}(\U_{i+1})$ for all $i$.
Then
$$g_{\U_1}^{\W_1}\cdots g_{\U_t}^{\W_t}=
g^{\W_t}_{\overline{\varphi}^{\W_t}(\U_1,\ldots ,\U_t)}.$$
\end{lemma}
\begin{proof}
We prove the result by induction on $t$. The result is clear for
$t=1$, so suppose that the result holds for $t-1$.
We have, using Proposition~\ref{p:phit}:
\begin{align*}
g_{\U_1}^{\W_1}\cdots g_{\U_t}^{\W_t} &=
(g_{\U_1}^{\W_1}\cdots g_{\U_{t-1}}^{\W_{t-1}}) g_{\U_t}^{\W_t} \\
&= g^{\W_{t-1}}_{\overline{\varphi}_{t-1}^{\W_{t-1}}(\U_1,\ldots ,\U_{t-1})} g_{\U_t}^{\W_t} \\
&=g^{\W_t}_{\F_{\U_t}(\overline{\varphi}_{t-1}^{\W_{t-1}}(\U_1,\ldots ,\U_{t-1})\amalg \U_t)} \\
&=g^{\W_t}_{\overline{\varphi}_t^{\W_t}(\U_1,\ldots ,\U_t)},
\end{align*}
as required.
\end{proof}

\begin{lemma} \label{l:limitedfactor}
Let $\W$ be a wide subcategory of $\module\Lambda$, and let $\V$ be a support $\tau$-rigid object in $\C(\W)$. If
$$g^{\W}_{\V}=g_{\U_1}^{\W_1}\cdots g_{\U_t}^{\W_t}$$
is a factorization of $g^{\W}_{\V}$ as a composition of $t$ irreducible
maps, then $t$ is the number of indecomposable direct summands of $\V$.
\end{lemma}
\begin{proof}
By Lemma~\ref{l:irreducible}, the $\U_i$ are indecomposable objects of
$\C(\W)$. By Lemma~\ref{l:composition}, we have
$$g^{\W}_{\V}=g^{\W_t}_{\overline{\varphi}^{\W_t}(\U_1,\ldots ,\U_t)},$$
so
$$\V=\overline{\varphi}^{\W_t}(\U_1,\ldots ,\U_t),$$
and the result follows.
\end{proof}

We now prove our first main result of this section.

\begin{proposition} \label{p:factorbijection1}
Let $\W$ be a wide subcategory of $\module\Lambda$, and let
$\V$ be a support $\tau$-rigid object in $\C(\W)$ with $t$ indecomposable direct summands. Then there is a bijection between:
\begin{itemize}
\item[(a)] The set of $\tau$-exceptional sequences $(\U_1,\ldots ,\U_t)$ in $\W$ such that $\overline{\varphi_t}(\U_1,\ldots ,\U_t)=\V$;
\item[(b)] The set of factorizations of $g^{\W}_{\V}$ into
compositions of irreducible maps in $\mathfrak{W}$.
\end{itemize}
\end{proposition}
\begin{proof}
Given a sequence $(\U_1,\ldots ,\U_t)$ as in (a), set
$\W_t=\W$ and $\W_i=J_{\W_{i+1}}(\U_{i+1})$ for all $i$.
Then the composition
$g_{\U_1}^{\W_1}\cdots g_{\U_t}^{\W_t}$ is well-defined by
Lemma~\ref{l:welldefined} and equals $g^{\W}_{\V}$ by the
assumption in (a) and Lemma~\ref{l:composition}.
By Lemma~\ref{l:irreducible}, each map $g_{\U_i}^{\W_i}$ is
irreducible in $\mathfrak{W}$.

Any factorization as in (b) must have $t$ factors by Lemma~\ref{l:limitedfactor},so must have form $g_{\U_1}^{\W_1}\cdots g_{\U_t}^{\W_t}=g^{\W}_{\V}$. Given such a factorization, each $\U_i$ is indecomposable by Lemma~\ref{l:irreducible} and
$\V=\overline{\varphi_t}(\U_1,\ldots ,\U_t)$ by Lemma~\ref{l:composition}.
Furthermore, $(\U_1,\ldots ,\U_t)$ is a $\tau$-exceptional
sequence by Lemma~\ref{l:welldefined}.

It is clear that these two constructions are inverses of each other,
and hence give bijections between the sets in (a) and (b) as required.
\end{proof}

Recall, from \cite{bm}, that an {\em ordered support $\tau$-tilting object} in
$\C(\Lambda)$ is a sequence $$(\T_1, \dots, \T_n)$$ of indecomposable
support $\tau$-rigid objects in $\C(\Lambda)$ with the property that $\amalg_i \T_i$
is a support tilting object. 

\begin{theorem} \label{t:factorbijection2}
Let $\W$ be a wide subcategory of $\module\Lambda$ and $\V$ a
support $\tau$-rigid object in $\C(\W)$ with $t$ indecomposable
direct summands. Then the bijection $\varphi_t$ induces a
bijection between the following sets:
\begin{itemize}
\item[(a)] Factorisations of $g_{\V}^{\W}$ into compositions
of irreducible maps in $\mathfrak{W}$;
\item[(b)] Ordered decompositions of $\V$ into direct sums
of indecomposable objects in $\C(\W)$.
\end{itemize}
\end{theorem}
\begin{proof}
By Proposition~\ref{p:factorbijection1}, there is a bijection between
the set in (a) and the set of $\tau$-exceptional sequences
$(\U_1,\ldots, \U_t)$ in $\W$ such that
$\overline{\varphi_t}^{\W}(\U_1,\ldots ,\U_t)=\V$.
The result now follows from Theorem~\ref{t:bmbijection}.
\end{proof}

\section{Example}\label{examp}

In this section we consider the following example.
Let $Q$ be the quiver 
$$
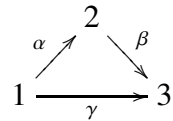
\xymatrix@=5mm{
& 2  \ar^{\beta}[dr] & \\
1  \ar^{\alpha}[ur]  \ar_{\gamma}[rr]& & 3
} 
$$
and consider the algebra $\Lambda = kQ/I$ where $I$ is the ideal generated by the path 
$\beta \alpha$. The AR-quiver of $\module \Lambda$ is
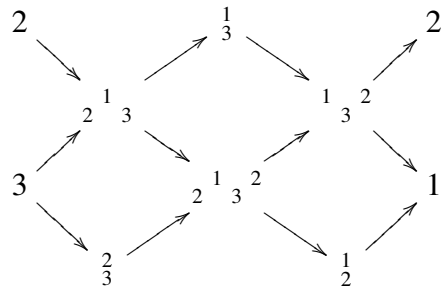
$$\xymatrix@=4mm{
2  \ar[dr] & &  {\begin{smallmatrix} 1 \\ 3 \end{smallmatrix}}  \ar[dr]& &  2 \\
&  {\begin{smallmatrix} & 1 &  \\ 2 & & 3 \end{smallmatrix}} \ar[ur] \ar[dr]& &  {\begin{smallmatrix} 1&  &2  \\  & 3&  \end{smallmatrix}}\ar[ur] \ar[dr] & \\
 3 \ar[ur] \ar[dr] & & {\begin{smallmatrix} & 1 && 2  \\  2 && 3 & \end{smallmatrix}} \ar[ur] \ar[dr] & &  1\\
&  {\begin{smallmatrix} 2 \\ 3 \end{smallmatrix}} \ar[ur] & &  {\begin{smallmatrix}  1   \\ 2  \end{smallmatrix}} \ar[ur]& 
}
$$
where the notation indicates which simple modules occur in the radical layers of the module, so $N = {\begin{smallmatrix} & 1 && 2  \\  2 && 3 & \end{smallmatrix}}$ is a module of length 4, of radical length 2, and with top isomorphic to the direct sum of the simple modules corresponding to vertices 1 and 2.
 
Figure~\ref{f:wlambda} gives an illustration of the category $\mathfrak{\W}_{\Lambda}$.
The vertices are the sets of indecomposable objects in each wide subcategory.
A non-identity morphism $g_{\T}^{\W} :\W \to \W'$ (so that $\T$ is an indecomposable
support $\tau$-rigid object in $\C(\W)$ and $J_{\W}(T)=\W'$) is
shown as an arrow between $\W$ and $\W'$ labelled by $T$.
When $P$ is projective in $\W$ we have $J_{\W}(P) = J_{\W}(P[1])$, and
there are two corresponding maps, $g_{P}^{\W}$ and $g_{P[1]}^{\W}$ from
$\W$ to $J_{\W}(P)$; in this case we draw a doubled arrow labelled only by $P$.
Wide subcategories of rank $1$ have generally been shown more than once in the figure, and the corresponding vertices should be identified.

\begin{sidewaysfigure}
\begin{center}
\xymatrix@C=0.7cm@R=1.4cm{
& & & & & & & & \\
& & & & & & & & \\
& & & & & & & & \\
& & & & & & & & \\
& & & & & & & & \\
& & & & & & & & \\
& & & & & & & & \\
\{ {\begin{smallmatrix} 2 \\ 3 \end{smallmatrix}} \} & &   \{ {\begin{smallmatrix} 1 \\ 3 \end{smallmatrix}} \} & \{ {\begin{smallmatrix} & 1 &  \\ 2 & & 3 \end{smallmatrix}} \} & & \{ {\begin{smallmatrix} 1 & & 2 \\ & 3 & \end{smallmatrix}} \} & \{   {{\begin{smallmatrix} 2 \end{smallmatrix}}} \} &
 &  \{ {\begin{smallmatrix} 1 \\ 3 \end{smallmatrix}} \} \\
\{ {\begin{smallmatrix} 1 \\ 3 \end{smallmatrix}} \} &   & \{ {\begin{smallmatrix} 1 \\ 3 \end{smallmatrix}} , {\begin{smallmatrix} 2 \\ 3 \end{smallmatrix}} \}
\ar@{=>}|(.5){\circlesign{\begin{smallmatrix} 2 \\ 3 \end{smallmatrix}}}[ll] 
\ar@{=>}|(.5){\circlesign{\begin{smallmatrix} 1 \\ 3 \end{smallmatrix}}}[ull]
 &   & \{ {\begin{smallmatrix} & 1&\\ 2 & & 3 \end{smallmatrix}} , {\begin{smallmatrix} 1 \\ 3 \end{smallmatrix}}, {\begin{smallmatrix} 1 & & 2 \\ & 3 & \end{smallmatrix}}, {\begin{smallmatrix} & 1 && 2  \\  2 && 3 & \end{smallmatrix}}, {\begin{smallmatrix} 2 \end{smallmatrix}}  \}
   \ar@{=>}|(.5){\circlesign{\begin{smallmatrix} & 1 & \\ 2 & & 3 \end{smallmatrix}}}[urr] \ar|(.5){\circlesign{\begin{smallmatrix} 2  \end{smallmatrix}}}[ur] \ar|(.5){\circlesign{\begin{smallmatrix}  1  \\ 3 \end{smallmatrix}}}[ul] \ar@{=>}|(.78){\circlesign{\begin{smallmatrix} & 1 & & 2 \\ 2 & & 3 & \end{smallmatrix}}}[ull] &  & \{ {\begin{smallmatrix} 1 \\ 2 \end{smallmatrix}}, {\begin{smallmatrix} 1 \\ 3 \end{smallmatrix}} \}  
   \ar@{=>}|(.5){\circlesign{\begin{smallmatrix} 1 \\ 3 \end{smallmatrix}}}[rr]  \ar@{=>}|(.5){\circlesign{\begin{smallmatrix} 1 \\ 2 \end{smallmatrix}}}[urr] & &  \{ {\begin{smallmatrix} 1 \\ 2 \end{smallmatrix}} \}\\
\{ {\begin{smallmatrix} 2 \\ 3 \end{smallmatrix}} \} & & & & & & & & \{ {\begin{smallmatrix} 3 \end{smallmatrix}} \} \\ 
\{ {\begin{smallmatrix} 1 & & 2 \\ & 3 & \end{smallmatrix}} \} & & \{ {\begin{smallmatrix} 2 \\ 3 \end{smallmatrix}} , {\begin{smallmatrix} 1 & & 2 \\ & 3  &\end{smallmatrix}}, {\begin{smallmatrix} 1  \end{smallmatrix}} \}
 \ar|(.5){\circlesign{\begin{smallmatrix} 1 \end{smallmatrix}}}[ll] 
 \ar@{=>}|(.5){\circlesign{\begin{smallmatrix}  2 \\ 3   \end{smallmatrix}}}[dll] 
 \ar@{=>}|(.5){\circlesign{\begin{smallmatrix} 1 &  & 2\\ & 3 & \end{smallmatrix}}}[ull]
 &  & \ind \Lambda  
 \ar|(.5){\circlesign{\begin{smallmatrix} 1 \\ 2 \end{smallmatrix}}}[uu]  \ar|(.5){\circlesign{\begin{smallmatrix}  & 1 & & 2  \\  2 && 3 &  \end{smallmatrix}}}[uull]  \ar|{\circlesign{\begin{smallmatrix} 1 \end{smallmatrix}}}[uurr]  \ar|{\circlesign{\begin{smallmatrix} 1 \\ 3 \end{smallmatrix}}}[rr] 
 \ar|(.5){\circlesign{\begin{smallmatrix} 2 \end{smallmatrix}}}[ll]  
 \ar@{=>}|(.5){\circlesign{\begin{smallmatrix} 3 \end{smallmatrix}}}[ddrr]  
 \ar@{=>}|(.5){\circlesign{\begin{smallmatrix}  2  \\ 3 \end{smallmatrix}}}[dd]  \ar@{=>}|(.5){\circlesign{\begin{smallmatrix} & 1 & \\ 2 & & 3\end{smallmatrix}}}[ddll]
& & \{ {\begin{smallmatrix} 3 \end{smallmatrix}} , {\begin{smallmatrix}  & 1 &  \\ 2 & & 3  \end{smallmatrix}}, {\begin{smallmatrix} 1  \\ 2\end{smallmatrix}} \} 
\ar@{=>}|(.5){\circlesign{\begin{smallmatrix} & 1 & \\ 2 & & 3 \end{smallmatrix}}}[urr] \ar@{=>}|(.5){\circlesign{\begin{smallmatrix} 3 \end{smallmatrix}}}[drr] 
\ar|(.5){\circlesign{\begin{smallmatrix} 1  \\ 2 \end{smallmatrix}}}[rr] & & \{ {\begin{smallmatrix} & 1 & \\2 & & 3 \end{smallmatrix}} \}  \\
\{ {\begin{smallmatrix} 1 \end{smallmatrix}} \} & & & & & & & & \{ {\begin{smallmatrix}  1  \\ 2  \end{smallmatrix}} \} \\ 
 \{ {\begin{smallmatrix} 3  \end{smallmatrix}} \}
&  & \{ {\begin{smallmatrix}  3 \end{smallmatrix}} , {\begin{smallmatrix} 2 \\ 3 \end{smallmatrix}},
 {\begin{smallmatrix}  2 \end{smallmatrix}} \} 
 \ar@{=>}|(.5){\circlesign{\begin{smallmatrix} 2 \\ 3 \end{smallmatrix}}}[ll] 
 \ar@{=>}|(.5){\circlesign{\begin{smallmatrix}  3 \end{smallmatrix}}}[d] 
 \ar|(.5){\circlesign{\begin{smallmatrix} 2  \end{smallmatrix}}}[dll]&  & 
 \{ {\begin{smallmatrix} 3 \end{smallmatrix}} , {\begin{smallmatrix} 1 \\ 3 \end{smallmatrix}}, {\begin{smallmatrix} 1  \end{smallmatrix}} \} 
 \ar@{=>}|(.5){\circlesign{\begin{smallmatrix} 1 \\ 3 \end{smallmatrix}}}[dl] 
 \ar|(.5){\circlesign{\begin{smallmatrix} 1 \end{smallmatrix}}}[d]
 \ar@{=>}|(.5){\circlesign{\begin{smallmatrix}  3 \end{smallmatrix}}}[dr]
   &  & \{ {\begin{smallmatrix} 2 \end{smallmatrix}}, {\begin{smallmatrix} 1 \\ 2 \end{smallmatrix}}, {\begin{smallmatrix} 1 \end{smallmatrix}} \} 
   \ar|(.5){\circlesign{\begin{smallmatrix} 1 \end{smallmatrix}}}[drr] 
   \ar@{=>}|(.5){\circlesign{\begin{smallmatrix} 2 \end{smallmatrix}}}[rr]
   \ar@{=>}|(.5){\circlesign{\begin{smallmatrix} 1 \\2 \end{smallmatrix}}}[d]  & &
   \{ {\begin{smallmatrix}   1  \end{smallmatrix}} \}  & \\
\{{\begin{smallmatrix}  2\\  3  \end{smallmatrix}} \}  &  &  \{ {\begin{smallmatrix}  2    \end{smallmatrix}} \} & \{ {\begin{smallmatrix}  3  \end{smallmatrix}} \} & 
\{{\begin{smallmatrix}  1 \\ 3 \end{smallmatrix}} \} &  
 \{ {\begin{smallmatrix}  1    \end{smallmatrix}} \}  &  \{ {\begin{smallmatrix}   2  \end{smallmatrix}} \} & & \{ {\begin{smallmatrix}   1\\ 2  \end{smallmatrix}} \} 
}
\caption{Illustration of the category $\mathfrak{\W}_{\Lambda}$
}\label{f:wlambda}
\end{center}
\end{sidewaysfigure}
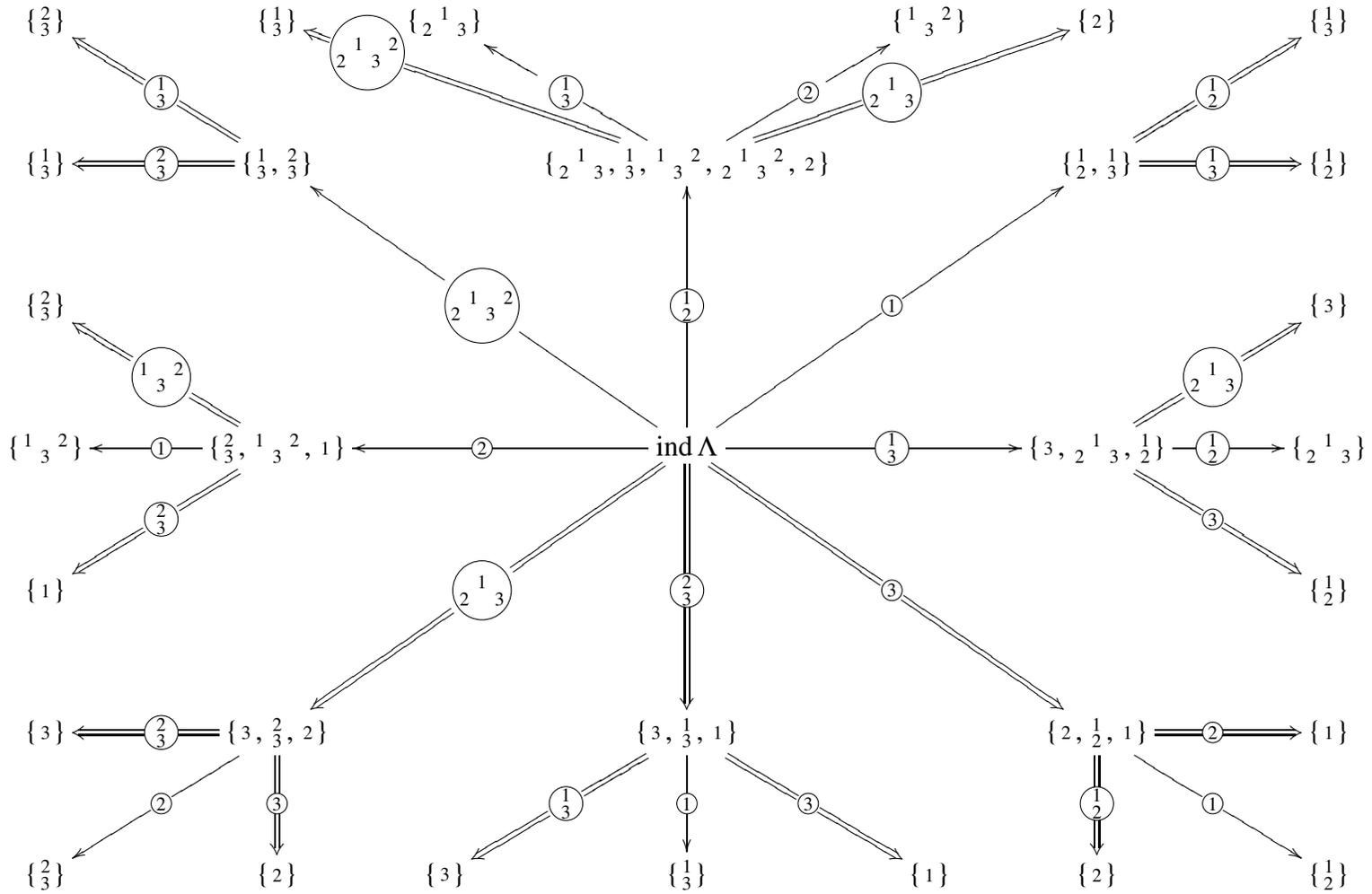

% i.e.
%$$F_{\U_t}^{\W_t}\cdots F_{\U_2}^{\W_2}(\U_1) \amalg
%F_{\U_t}^{\W_t}\cdots F_{\U_3}^{\W_3}(\U_2) \amalg
%\cdots \amalg \U_t = \V.$$

%Associated to $\SS$ is the wide subcategory $\W_{\SS}$ which can be obtained as follows:
%First let $\W_{\SS}^t = J(\U_t)$, and then for $s= 1, \dots, t-1$
%let $\W_{\SS}^{t-s} = J_{\W_{\SS}^{t-s+1}}(\U_{t-s})$.
%Then  $\W_{\SS} = \W_{\SS}^1$. 

%Recall that for a $\tau$-rigid object $\U$ in $\C(\Lambda)$, we proved in Section \ref{bi} that there are bijections  
%$$\{X \in \ind(\C(\Lambda)) \mid X \amalg \U \text{ } \tau\text{-rigid}\} 
%\setminus \ind \U$$ 

%$$\E_{\U} \downarrow \text{    } \uparrow \F_{\U}$$ 

% $$\{X \in \ind(\C(J(\U)) \mid X \text{ } \tau\text{-rigid} \}$$ 
 
%Note that by repeatedly applying Theorem \ref{main-comp} and these bijections we have 
%that $$\W_{\SS} = J(\U')$$ where
%$$\U' = \U_t \amalg \F_{\U_t} (\U_{t-1}) \amalg \dots \amalg  \F_{\U_t}  \F_{\U_{t-1}} \dots  \F_{\U_2} (\U_1))$$

%\begin{lemma}
%With notation as above, the map 
%\end{lemma}

%A signed $\tau$-rigid sequence $\SS$ can hence be interpreted as giving 
%a factorization of the map $g = g$ 

%\begin{proposition}
%\begin{itemize}
%\item[(a)] A map $f_T \colon \W_1 \to W_2= J(T)$ in $\mathbb{W}_T$ is irreducible
%if and only if $T$ is an indecomposable $\tau$-rigid object in $\C(\W_1)$.
%\item[(b)] Let $f_T \colon  \W_1 \to W_2= J(T)$ be an arbitrary map in $\C(\W_1)$.
%Then $f_T$ is a composition of $t$ irreducible maps, where $r= r(\W_1) - r(\W_2)$.
%\end{itemize}
%\end{proposition}

%\begin{proof}
%\end{proof}

\end{document}